\documentclass[a4paper,12pt]{article}
\usepackage[centertags]{amsmath}
\usepackage{amsfonts}
\usepackage{amssymb}
\usepackage{amsthm}
\usepackage{amsmath}
\usepackage{dsfont}
\usepackage{graphicx} 
\usepackage{tikz, subfigure}
\usepackage{pstricks-add}
\usepackage{verbatim}

\usepackage[latin1]{inputenc}
\usepackage{tikz}
\definecolor {processblue}{cmyk}{0.96,0,0,0}
\usepackage{amsfonts,graphicx,amsmath,amssymb,hyperref,color}
\usepackage[english]{babel}
\usepackage{authblk}

\newtheorem{Theorem}{Theorem}[section]
\newtheorem{Proposition}{Proposition}[section]
\newtheorem{Lemma}{Lemma}[section]
\newtheorem{Corollary}{Corollary}[section]

\newtheorem{Definition}{Definition}[section]

\newtheorem{Example}{Example}[section]

\newtheorem{Remark}{Remark}[section]

\newtheorem{Question}{Question}[section]

\AtEndDocument{\bigskip{\footnotesize%
  \textsc{Department of Mathematics, Utrecht University, Fac Wiskunde en informatica and MRI, Budapestlaan 6, P.O. Box 80.000, 3508 TA Utrecht, The Netherlands} \par
  \textit{E-mail address:} \texttt{K.Jiang1@uu.nl,\, k.dajani1@uu.nl.} \par
}}

\usepackage{lipsum}

\makeatletter
\newcommand*{\rom}[1]{\expandafter\@slowromancap\romannumeral #1@}
\makeatother

\begin{document}
\title{xxxx}
\date{}
 \title{Subshifts of finite type and self-similar sets}
\author{Kan Jiang and Karma Dajani\thanks{Karma Dajani  is the corresponding author.} }
\maketitle{}
\begin{abstract}
Let  $K\subset \mathbb{R}$ be a self-similar set generated by  some iterated function system. In this paper we prove, under some assumptions, that $K$ can be identified  with a subshift of finite type. With this identification,  we can calculate the Hausdorff dimension of $K$   as well as  the set of elements in $K$ with  unique codings using the machinery of Mauldin and Williams \cite{MW}. 
We give three different applications of our main result. Firstly, we calculate the Hausdorff dimension of the set of points  of $K$ with multiple codings. Secondly, in the setting of $\beta$-expansions,  when  the set of all the unique codings is not a subshift of finite type,  we can calculate in some cases the Hausdorff dimension of the univoque set.  Motivated by this application, we prove that the set of all the unique codings is  a subshift of finite type if and only if it is a sofic shift. This equivalent condition was not mentioned by de Vries and Komornik \cite[Theorem 1.8]{MK}.  Thirdly, for the doubling map with  asymmetrical holes, we give a sufficient condition such that the   survivor set  can be identified with a subshift of finite type.
 The third application partially  answers a problem posed by  Alcaraz Barrera \cite{Barrera}. 
\end{abstract}

\section{Introduction}
Let  $F=\{f_i\}_{i=1}^{m}$ be the contractive similitudes defined on $\mathbb{R}$. Hutchinson \cite{Hutchinson} proved  that there exists a unique non-empty  compact set $K$ satisfying the following equation
$$K=\cup_{i=1}^{m}f_{i}(K).$$
We call $K$ the self-similar set for  the iterated function system (shortly IFS)$\{f_i\}_{i=1}^{m}$. 
We say $\{f_{i}\}_{j=1}^{m}$ satisfies the open set condition (OSC) \cite{Hutchinson}  if there exists a non-empty bounded open set $O\subseteq \mathbb{R}$ such that
\[f_i(O)\cap f_{j}(O)=\varnothing,\, i\neq j\]
and $f_j(O)\subseteq O$ for all  $1\leq j\leq m$. 
With the open set condition, the Hausdorff dimension of $K$, which coincides with the similarity dimension that is the  unique solution $s$ of the equation $\sum_{i=1}^{m}|r_i|^s=1$, can be easily calculated. 
Here $\{r_i\}_{i=1}^{m}$ are the similarity ratios of $\{f_i\}_{i=1}^{m}$. 

Clearly, for any $x\in K$, there exists a sequence $(i_n)\in \{1,\cdots, m\}^{\mathbb{N}}$
such that
\[x=\lim_{n\to \infty}f_{i_1}\circ \cdots\circ f_{i_n}(0).\] 
We call $(a_n)$ a coding of $x$. Generally, if the IFS fails the open set condition, then for some points of $K$  they may have multiple codings. 
 If $x\in K$ has a unique coding,  then we call $x$ a
univoque point. The set of univoque points is called the univoque
set, and we denote it by $U_{F},$ i.e.,
\begin{align*}
U_{F}:=\Big\{x\in K: &\textrm{ there exists a unique } (i_{n})_{n=1}^{\infty}\in \{1,\ldots,m\}^{\mathbb{N}} \textrm{ satisfying }\\
& x=\lim_{n\to \infty}f_{i_1}\circ \cdots\circ f_{i_n}(0)\Big\}.
\end{align*}
Let $\widetilde{U}_{F}$ be the set of all the unique codings with respect to $\{f_i\}_{i=1}^{m}.$

Calculating the Hausdorff dimension of a self-similar set is a crucial problem in  fractal geometry. Usually, it is difficult to find the dimension of a self-similar set, especially  when  serious overlaps occur, see \cite{LauNgai, Hochman, NW, BBB} and references therein. In this paper we offer an effective methodology which enables us to find the Hausdorff dimension of many self-similar sets.   We give a brief introduction of our idea here.  Firstly, we assume that  for any $f_i(K)\cap f_j(K)\neq \emptyset$, $i\neq j$,  $f_i(K)\cap f_j(K)$ is the union of some exact overlaps (see Definition \ref{exact overlap}) and that the endpoints of each $f_i(K), 1\leq i\leq m,$  have periodic orbits  (see Definition \ref{Periodic orbits}).
The periodicity of the orbits of the endpoints allows us to give a Markov partition (see the definition in \cite{app}) of the self-similar set $K$. This step is essential as dynamically we transform the original full shift into  a subshift of finite type, which allows us to view  our attractor as a graph-directed self-similar set (\cite{MW}). The graph-directed self-similar set satisfies the open set condition due to the Markov property of the partition. As such we can calculate explicitly the Hausdorff dimension of $K$ (\cite[Theorem 3]{MW}).  Similar idea enables us to find the Hausdorff dimension of the univoque set as well. Analogous results still hold in higher dimensions. 

The first application of our method is the investigation of  the following set, 
$$U_{k}:=\{x\in K:x\,\, \mbox{ has exactly} \,\, k \,\,  \mbox{codings}\}, k=1, 2, \cdots, \aleph_0.$$
Considering this set can assist us in a better understanding of the coding space of  $K$.
In this paper we investigate only  one example, see Theorem \ref{k=1}. For the generalized results, see \cite{DJKL1, DJKL}. We give a simple introduction to our result. 
Suppose  that $K$ is the attractor of the following IFS,
$$\left\{f_{1}(x)=\lambda x, f_{2}(x)=\lambda x+2\lambda,  f_{3}(x)=\lambda x+3\lambda-\lambda^2,f_4(x)=\lambda x+1-\lambda\right\}.$$
where $0<\lambda<\dfrac{5-\sqrt{21}}{2}$.

Our main result states that
for any $k\geq 1$,
$\dim_{H}(U_{2^k})=\dim_{H}(U_1)$. Moreover, $U_i=\emptyset, i\neq 2^{s}, s\geq1$, and $U_{\aleph_0}=\emptyset$. 
As a corollary,  we have that
any  $x\in K$ has either exactly $2^k$ expansions, $k\geq 0$,  or has uncountably many expansions. 
This example illustrates the key ideas which can analyze $U_{k}$.

The second application is in the setting of $\beta$-expansions.  We are able to calculate $\dim_{H}(U_{F})$ for some cases for which $\tilde{U}_{F}$ is not a subshift finite type. We first introduce some notation.
Let $\beta\in(1, 2)$ and $x\in \mathcal{A}_{\beta}=[0, (\beta-1)^{-1}]$, we call a sequence $(a_n)_{n=1}^{\infty}\in\{0,1\}^{\mathbb{N}}$ a $\beta$-expansion of $x$ if $$x=\sum_{n=1}^{\infty}\dfrac{a_n}{\beta^n}.$$
Sidorov \cite{Sidorov} proved that  Lebesgue almost every point has uncountable expansions. 
We  use $\tilde{U}_{\beta}$ and $U_{\beta}$ to denote  the  set of unique $\beta$-expansions  and the  corresponding univoque set. The  dynamical approach is an excellent tool which can generate  $\beta$-expansions  effectively. Define $T_0(x)=\beta x, T_1(x)=\beta x-1$, see Figure 1. 
\begin{figure}[h]\label{figure1}
\centering
\begin{tikzpicture}[scale=3]
\draw(0,0)node[below]{\scriptsize 0}--(.382,0)node[below]{\scriptsize$\frac{1}{\beta}$}--(.618,0)node[below]{\scriptsize$\frac{1}{\beta(\beta-1)}$}--(1,0)node[below]{\scriptsize$\frac{1}{\beta-1}$}--(1,1)--(0,1)node[left]{\scriptsize$\frac{1}{\beta-1}$}--(0,.5)--(0,0);
\draw[dotted](.382,0)--(.382,1)(0.618,0)--(0.618,1);
\draw[thick](0,0)--(0.618,1)(.382,0)--(1,1);
\end{tikzpicture}\caption{The dynamical system for $\{T_{0},\,T_{1}\}$}
\end{figure}
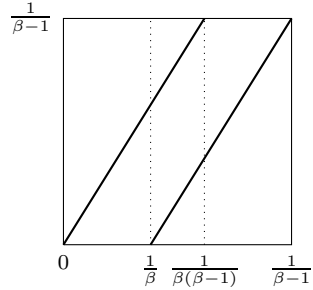
All possible  $\beta$-expansions can be generated via these two maps, see \cite{KarmaCor, KM}. 
We call the common domain of $T_0$ and $T_1$, i.e. $[\beta^{-1}, \beta^{-1}(\beta-1)^{-1}]$,  the switch region.
Much work has been done regarding $U_{\beta}$, see for example \cite{GS,MK}.  
One of the  motivations of this paper is to  continue the investigation of  the Hausdorff dimension of $U_{\beta}$ from the dynamical point of view.  In \cite{SKK}, it is proved that if the greedy orbit of $1$ hits $( \beta^{-1}, \beta^{-1}(\beta-1)^{-1})$, then $\tilde{U}_{\beta}$ is a subshift of finite type.  Schmeling \cite{Schmeling} proved that for almost every $\beta\in(1,2)$, the greedy orbit of $1$ is dense, which implies that for almost every $\beta$, $\tilde{U}_{\beta}$ is a subshift of finite type. 
 Hence, we can take advantage of Mauldin and Williams' result  to calculate $\dim_{H}(U_{\beta})$.
On the other hand, if $\beta$ is a Pisot number, then the greedy orbit of $1$ may not hit the interior of the switch region. In this case, the idea in \cite{SKK} cannot be implemented, and we have to find a new approach.

De Vries and  Komornik proved in \cite{MK} that $\tilde{U}_{\beta}$ is a subshift of finite type  if and only if $\beta\in (1,2)\setminus \overline{\mathcal{U}}$, where $\mathcal{U}$ is the set of $\beta$ for which the $\beta$-expansion of 1 is unique.  This result allowed them to calculate the Hausdorff dimension of $U_{\beta}$ if $\tilde{U}_{\beta}$ is a subshift of finite type. In this paper, we  give a necessary and sufficient condition which can characterize when $\tilde{U}_{\beta}$ is a subshift of finite type in terms of the quasi-greedy orbit of $1$, see Theorem \ref{Dynamical}. Using this result  together with \cite[Theorem 2.4]{SKK}, we can give an algorithm to find $\dim_{H}(U_{\beta})$ in this case.  In some cases when the greedy orbit of $1$ is  eventually periodic, we are able to calculate $\dim_{H}(U_{\beta})$ even when $\tilde{U}_{\beta}$ is not a subshift of finite type. Moreover, we prove that $\tilde{U}_{\beta}$ is a subshift of finite type if and only if it is a sofic shift. 
As such we  generalize de Vries and Komornik's result  concerning the calculation of  $\dim_{H}(U_{\beta})$. 

Some dimensional problems  in the open dynamical systems are involved in the  third application of our main result. We only study the doubling map with holes. Let $T(x)=2 x \mod 1$ be the doubling map defined on $[0,1)$. Given any $(a,b)\subset [0,1)$, referred to as a hole, let 
$$J(a,b)=\{x\in [0,1):T^{n}(x)\notin (a,b), \forall\, n\geq 0\}.$$
Glendinning and Sidorov \cite{GGG} gave a complete description of $J(a,b)$. More precisely,  given any  hole $(a,b)\subset (0,1)$ they can characterize when $J(a,b)$ is  of zero or positive Hausdorff dimension.  Clark \cite{Clark}
made use of the same techniques and  obtained similar results when $T$ is the greedy  map. 
However, both of these papers do not offer an approach to finding the exact Hausdorff dimension of $ J(a,b)$. The main idea of this paper  allows us to calculate the Hausdorff dimension of the attractor in certain cases. 
Moreover, we partially answer  Alcaraz Barrera's question \cite{Barrera}, i.e. we prove that if the orbits of $a$ and $b$ are  eventually periodic,  and the irreducibility condition is satisfied, no matter where the hole is located,    $J(a,b)$ is  measure theoretically isomorphic  to a subshift of finite type. 

This paper is organized as follows. In section 2, we state our main results and give the proofs. We also investigate an example concerning the  Hausdorff dimension of $U_k$. 
 In section 3 and section 4,  we study the case of  $\beta$-expansions. Firstly  we give, from dynamical perspective, an equivalent  statement of a result of  de Vries and Komornik \cite[Theorem 1.8]{MK}.
Then we implement similar idea which is utilized in section 2, and calculate $\dim_{H}(U_{\beta})$  when the  greedy orbit of $1$ is eventually periodic. In section 5, we  partially answer one problem posed by  Alcaraz Barrera \cite{Barrera}.  
 
\section{Hausdorff dimension of $K$ and $U_{F}$}
\subsection{Dimension of $K$}
In this section we shall state the main results of our paper. To begin with, we introduce some basic notation. 
Define $T_{j}(x):=f_{j}^{-1}(x)=(x-a_{j})r_{j}^{-1}$ for $x\in f_j(K)$ and $1\leq j\leq m$, where 
$f_{i}(x)=r_j x+a_j, 1\leq j\leq m$. 
We denote the concatenation $T_{i_{n}}\circ \ldots \circ T_{i_{1}}(x)$ by $T_{i_1\ldots i_n}(x)$. The following lemma was proved in \cite[Lemma 2.1]{SKK}. The main hypothesis in \cite[Lemma 2.1]{SKK} is that $K$ is an interval. We emphasize that all the self-similar sets in this paper are not necessarily  intervals. The following lemma is still true if $K$ is a general self-similar set.
\begin{Lemma}
\label{Dynamical lemma}
Let $x\in K.$ Then $(i_{n})_{n=1}^{\infty}\in\{1,\ldots,m\}^{\mathbb{N}}$ is a coding for $x$ if and only if $T_{i_{1}\ldots i_{n}}(x)\in K$ for all $n\in\mathbb{N}.$
\end{Lemma}
Motivated by Lemma \ref{Dynamical lemma}, we may define the orbits of the points of $K$.
\setcounter{Definition}{1}
\begin{Definition}\label{Orbit set}
Let $x\in K$ with  a coding  $(i_n)_{n=1}^{\infty}$, we call  the set $$\{T_{i_{1}\ldots i_{n}}(x): n\geq 0\}$$ an  \textbf{orbit} set of  $x$, where $T_{i_0}(x)=x.$
\end{Definition}
It is easy to see that for different codings, the orbits of $x$ may be distinct.
In terms of Lemma \ref{Dynamical lemma}, we have the following proposition which 
  is a straightforward consequence. 
\setcounter{Proposition}{2}
\begin{Proposition}
\label{non unique prop}
Let $x\in K$.  There exist $(i_{n})_{n=1}^{N}\in\{1,\ldots,m\}^{N}$ and distinct $k,l\in \{1,\ldots,m\}$ satisfying $ T_{i_{1}\cdots i_{N}k}(x)\in K$ and $T_{i_{1}\cdots i_{N}l}(x)\in K$ if and only if  $x\notin U_{F}.$
\end{Proposition}
Let $I_{j}=f_{j}(K), 1\leq j\leq m$.  The following reformulation of $U_{F}$ is a consequence of Proposition \ref{non unique prop}:
\begin{equation}
\label{univoque reformulation111}
U_{F}=\Big\{x\in K:  \textrm{ any orbit of  } x \textrm{ does not hit }\bigcup_{k\neq l} (I_{k}\cap I_{l})\Big\}.
\end{equation}
I.e., if $x\in U_F$, then  for any orbit of  $x$, denoted by  $\{T_{i_{1}\cdots i_{n}}(x):n\geq 0\}$,  $$T_{i_{1}\cdots i_{n}}(x)\cap( \bigcup_{k\neq l} (I_{k}\cap I_{l}))=\emptyset$$ holds for any $n\geq 0.$ In other words, if $x\in U_F$, then  any orbit of $x$ never falls into $\bigcup_{k\neq l} (I_{k}\cap I_{l}).$
Now we give some important definitions. 
Let $A\subset \mathbb{R}$ be a bounded set. We call the minimal and maximal elements of $conv(A)$  the endpoints of $A$, where $conv(A)$ denotes the convex hull of $A$.  
\setcounter{Definition}{3}
\begin{Definition}\label{Periodic orbits}
Let $a_1<a_2<\cdots<a_{2m}$ be  the endpoints of $\{f_j(K)\}_{j=1}^{m}$.  We say $a_i$ is a  \textbf{periodic point} if 
there exists a uniform constant $M>0$ such that for any orbit set of $a_i$, say $F$, the cardinality of $F$ is less than $M$. 
\end{Definition}
\begin{Definition}\label{exact overlap}
We call $f_{i_1\cdots i_n}(K)$ an \textbf{exact overlap} if there exists a different $$(j_1\cdots j_t)\in \{1,2,3\cdots m\}^{t}$$ for some $t\geq 1$ such that $f_{i_1\cdots i_n}(K)=f_{j_1\cdots j_t}(K)$.
\end{Definition}
\begin{Definition}\label{exact overlapping IFS}
We say  an IFS is an \textbf{exactly overlapping IFS}, if for any $f_i(K)\cap f_j(K)\neq \emptyset$, $i\neq j$,  $f_i(K)\cap f_j(K)$ is the union of some exact overlaps. More precisely, there exist $\mathbf{t}_q\in \{1,2\cdots,m\}^{k_q}, 1\leq q\leq p$  such that $f_{i}(K)\cap  f_{j}(K)=  \cup_{q=1}^{p} f_{\mathbf{t}_q}(K)$, where each $f_{\mathbf{t}_q}(K)=f_{i\,i_2i_3\cdots i_{k_q}}(K)=f_{j\,j_2j_3\cdots j_r}(K)$ for some $r$. 
\end{Definition}
Now we give an example to classify the above definitions.  The following example is from \cite{NW}.
\setcounter{Example}{6}
\begin{Example}\label{EX11}
Let $K$ be the self-similar set of the following IFS:
$$\left\{f_1(x)=\dfrac{x}{3},\, f_2(x)=\dfrac{x}{9}+\dfrac{8}{27},\, f_3(x)=\dfrac{x+2}{3}\right\}.$$
Set $J=[0,1]$. It is easy to check that  $f_{1}(J)=\left[0,\dfrac{1}{3}\right],f_{2}(J)=\left[\dfrac{8}{27},\dfrac{11}{27} \right],f_{3}(J)=\left[\dfrac{2}{3},1\right],$ and that
$$f_{1}(K)\cap f_{2}(K)\neq \emptyset, f_{2}(K)\cap f_{3}(K)= \emptyset, f_{1}(K)\cap f_{3}(K)= \emptyset,$$ see the following figure.

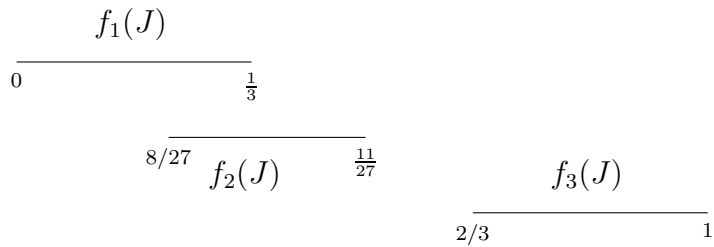
\begin{figure}[h]\label{figure1}
\centering
\begin{tikzpicture}[scale=5]
\draw(0,0)node[below]{\scriptsize $0$}--(.618,0)node[below]{\scriptsize$\frac{1}{3}$};
\draw(0.4,-0.2)node[below]{\scriptsize $8/27$}--(.918,-0.2)node[below]{\scriptsize$\frac{11}{27}$};
\draw(1.2,-0.4)node[below]{\scriptsize $2/3$}--(1.818,-0.4)node[below]{\scriptsize$1$};
\node [label={[xshift=1.5cm, yshift=0cm]$f_1(J)$}] {};
\node [label={[xshift=3cm, yshift=-2cm]$f_2(J)$}] {};
\node [label={[xshift=7.5cm, yshift=-2cm]$f_3(J)$}] {};
\end{tikzpicture}\caption{First iteration}
\end{figure}

Note that  $f_1(K)\cap f_2(K)=f_{1331}(K)\cup f_{1332}(K)\cup f_{1333}(K),$
where $$f_{1331}(K)=f_{211}(K), f_{1332}(K)=f_{212}(K), f_{1333}(K)=f_{213}(K), $$  i.e. $f_1(K)\cap f_2(K)$ is the union of some exact overlaps.
Hence the IFS is  exactly overlapping. 
\end{Example}
Suppose  the endpoints of $f_i(K),  1\leq i\leq m$, are periodic. 
Denote all these endpoints by 
$\mathcal{E}=\{a_1< a_2<a_{3}<\cdots <a_{2m}\}.$
Assume $I$ is the convex hull of $K$. We say $a_j$ and $ a_{j+1}$ are \emph{admissible pair} if there exists some $1\leq k\leq m$ such that 
$a_j, a_{j+1}\in f_{k}(I)$. 
It is easy to see that  for any  admissible pair $a_j, a_{j+1}$, there exists at least one $k$ such that $a_j, a_{j+1}\in f_{k}(I)$. Denote the smallest $k$ by $\alpha(k)$.  The following definition is motivated by the lazy map \cite{EJK}. 
\setcounter{Definition}{7}
\begin{Definition}\label{Def:28}
Let $F=\{f_i\}_{i=1}^{m}$ be an exactly overlapping IFS. 
Suppose  the endpoints of $f_i(K),  1\leq i\leq m$, are periodic. 
We define the lazy algorithm for this system as follows:
for any  admissible pair $a_j, a_{j+1}$  and any 
$x\in [a_j, a_{j+1}]\cap K,$
we implement $T_{\alpha(k)}$ on $x$. The orbit of $x$ generated by this algorithm is called the \textbf{lazy orbit}. 
 \end{Definition}
 We use Example \ref{EX11} to explain this definition. 
 Note that 
 $$T_1(x)=3x, T_2(x)=9x-8/3, T_3(x)=3x-2,$$
 and the endpoints of $\{f_i(K)\}_{i=1}^{3}$ are $\mathcal{E}=\{0,8/27, 1/3, 11/27, 2/3, 1\}$.
  The point $0$ has a unique coding, and therefore it has only one orbit. 
 Similarly, $11/27, 2/3, 1$ also have unique orbits. 
 It can be checked directly that all possible orbits of $8/27$ hit 
 $\{0,8/27, 2/3,8/9\}$ and that  all the orbits of $1/3$ hit $\{0, 1/3,1\}$. 
Note that $1/3$ is the right endpoint of $f_{1}(K)\cap f_2(K)$,  for any points  in 
$[8/27, 1/3]\cap K$, we can implement the expanding maps $T_1$ or $T_2$. 
However, according to our  lazy algorithm, we only choose $T_1$. For the points of 
$[1/3, 11/27]\cap K$, we can only choose $T_2$.  The pair $11/27$ and $2/3$  is not  admissible as there is no $1\leq k\leq 3$ such that $11/27, 2/3\in f_{k}(J).$

In the remaining part of this section,  we assume the IFS to be  exactly overlapping, and the orbits of  the endpoints of 
$\{f_i(K)\}_{i=1}^{m}$ are periodic. Moreover, in order to simplify the calculation, we  always choose the lazy algorithm which is defined in Definition \ref{Def:28}.
We assume the  periodic orbits of $\{a_1,a_2,\cdots, a_{2m}\}$ are 
$\{b_1,b_2,\cdots, b_{s+1}\}$,
i.e. 
$$K=\bigcup_{i=1}^{s}( [b_i, b_{i+1}]\cap K).$$
Generally,  we only know the orbit set $\{b_1,b_2,\cdots, b_{s+1}\}$ is finite. In other words, we may not know the exact relation between $2m$ and $s+1$. 
As  the endpoints of each $f_i(K)$ are periodic, 
for any $1\leq i\leq m$ and  $1\leq j\leq s$, we can find some $k$ such that 
$$T_i(A_j)=A_{i_1}\cup A_{i_2}\cup \cdots\cup A_{i_k},$$ 
where $A_j=[b_j, b_{j+1}]$. Hence,  the collection $\{B_i= [b_i, b_{i+1}]\cap K: 1\leq i\leq s\}$ is a Markov partition of $K$ (\cite{app}).
For some $j$, $(b_j, b_{j+1})\cap K$ may be empty,  we delete $B_j$ from the partition. 
We suppose without loss of  generality that   
$$K=\cup_{j=1}^{s}B_j,$$
and that $(b_j, b_{j+1})\cap K\not= \emptyset$ for all $j=1,\cdots, s$.
We call the sets $B_j=A_j\cap K, 1\leq j\leq s$ \emph{ the  blocks of the Markov partition}. We point out here that for the original definition of $T_i,1\leq i\leq m$, the domains are  $f_i(K)$. However, since $\{T_i \}_{i=1}^{m}$ are linear maps, we can extend the domains to $\mathbb{R}$.


The  following lemma is important as it allows us to  find the Markov partition of $K$. 
\setcounter{Lemma}{8}
\begin{Lemma}\label{Partition}
Let $\{f_i\}_{i=1}^{m}$ be an exactly overlapping IFS. 
Suppose  that the endpoints of each $f_i(K),1\leq i\leq m,$ are periodic. Then, for any $1\leq i\leq m$ and $1\leq j\leq s$, there exists some $k$ such that 
$$T_i(A_j\cap K)=(A_{i_1}\cap K)\cup (A_{i_2}\cap K)\cup \cdots\cup( A_{i_k}\cap K),$$ where
$$T_i(A_j)=A_{i_1}\cup A_{i_2}\cup \cdots\cup A_{i_k}.$$
\end{Lemma}
\begin{proof}
\begin{eqnarray*}
T_i(A_j\cap K )&=& T_i(A_j)\cap T_i(K)\\&=
&(A_{i_1}\cup A_{i_2}\cup \cdots\cup A_{i_k})\cap T_i(K)\\&=& (A_{i_1}\cap T_i(K))\cup (A_{i_2}\cap T_i(K)) \cup \cdots\cup (A_{i_k}\cap T_i(K))
\\&=&(A_{i_1}\cap K)\cup (A_{i_2}\cap K)\cup \cdots\cup( A_{i_k}\cap K)
\end{eqnarray*}
The first equality holds as  each $T_i$ is a bijection. The last equality  is due to the fact that for any $1\leq i\leq m$, $A_{i_j}\cap T_i(K)=A_{i_j}\cap K, 1\leq j\leq k$.
\end{proof}
 We denote  by $S$  the adjacency matrix of the Markov partition $\{B_i= [b_i, b_{i+1}]\cap K: 1\leq i\leq s\}$.
 
We exhibit  all the definitions above in the following example. 
\setcounter{Example}{9}
\begin{Example}\label{EX:10}
Let $q>\dfrac{3+\sqrt{5}}{2}$ be any real number and $\rho=q^{-1}$.  Consider the IFS
$$\left\{f_{0}(x)=\dfrac{x}{q}, f_1(x)=\dfrac{x+1}{q},f_{q}(x)=\dfrac{x+q}{q}\right\}.$$
The convex hull of $K$ is $E=[0,(1-\rho)^{-1}]$. Note that
$$f_{0}(E)=\left[0, \dfrac{\rho}{1-\rho}\right], f_1(E)=\left[\rho,\dfrac{2\rho-\rho^2}{1-\rho}\right], f_{q}(E)=\left[1, \dfrac{1}{1-\rho}\right].$$
It is easy to check that $f_0(E)\cap f_1(E)\neq \emptyset$ and $f_0(E)\cap f_q(E)=\emptyset, f_1(E)\cap f_q(E)=\emptyset$ as $q>\dfrac{3+\sqrt{5}}{2}$, see Figure 3. 
Then the considered  IFS is exactly overlapping  as we have $f_0(K)\cap f_1(K)= f_0\circ f_q(K)=f_1\circ f_0(K)$.   Moreover, we can check that the endpoints of $f_i(K), i=0,1,q$, hit  finite points. Hence, this IFS satisfies  the setting of our assumption.

\begin{figure}[h]\label{figure1}
\centering
\begin{tikzpicture}[scale=5]
\draw(0,0)node[below]{\scriptsize $0$}--(.618,0)node[below]{\scriptsize$\frac{1}{1-\rho}$};
\draw(0.3,-0.2)node[below]{\scriptsize $\rho$}--(.918,-0.2)node[below]{\scriptsize$\frac{2\rho-\rho^2}{1-\rho}$};
\draw(1.2,-0.4)node[below]{\scriptsize $1$}--(1.818,-0.4)node[below]{\scriptsize$\frac{1}{1-\rho}$};
\node [label={[xshift=1.5cm, yshift=0cm]$f_0(E)$}] {};
\node [label={[xshift=3cm, yshift=-2cm]$f_1(E)$}] {};
\node [label={[xshift=7.5cm, yshift=-2cm]$f_q(E)$}] {};
\end{tikzpicture}\caption{First iteration}
\end{figure}
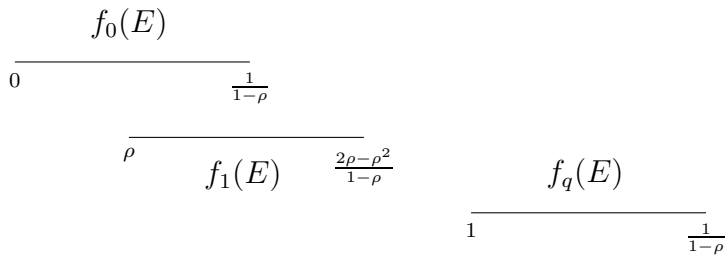

We partition  $K$ via $K=A\cup B \cup C \cup D$
where $$A=\left[0, \rho\right]\cap K, B=\left[\rho,\dfrac{\rho}{1-\rho}\right]\cap K, C=\left[\dfrac{\rho}{1-\rho}, \dfrac{2\rho-\rho^2}{1-\rho}\right]\cap K, D=\left[1,\dfrac{1}{1-\rho}\right]\cap K$$
A simple calculation yields that 
\begin{equation*}\label{MMM}
T_{0}(A)=A\cup B \cup C, T_{0}(B)=D,  T_{1}(C)=C\cup D, T_{q}(D)=A\cup B \cup C \cup D.
\end{equation*}
The adjacency matrix is
 $$
 S=\begin{pmatrix}
  1 & 1 & 1 & 0 \\
 0 & 0 & 0 & 1 \\ 
 0 & 0 & 1 & 1 \\
 1 & 1 & 1 & 1 \\
 \end{pmatrix}
.$$   
\end{Example}       

\medskip
Let $\Sigma$ be the subshift of finite type generated by $S$, i.e.
 $$\Sigma=\{(i_k)_{k=1}^{\infty}\in\{1,2,3,\cdots, s\}^{\mathbb{N}}:S_{i_k,i_{k+1}}=1\,\, \mbox{for any}\,\, k\}.$$
We now state one of the main results of this paper which allows us to calculate the Hausdorff dimension of many  fractal sets.
\setcounter{Theorem}{10}
\begin{Theorem}\label{graph-directed}
Let $K$ be the self-similar attractor of  an exactly overlapping IFS. If the endpoints of each $f_j(K)$ are periodic, then $K$ is a graph-directed self-similar set satisfying the open set condition. In particular, the Hausdorff dimension of $K$ can be calculated explicitly in terms of the adjacency matrix $S$.  Moreover, if the  associated directed graph of $S$ is strongly connected, then  the  associated dimensional Hausdorff measure of $K$ is  positive and finite. 
\end{Theorem}
\setcounter{Remark}{11}
\begin{Remark}
We may implement analogous ideas in higher dimensions. 
In \cite[Theorem 1.5]{Hochman} Hochman  proved the following result: 
let the IFS  on  $\mathbb{R} $ defined by algebraic parameters, that is, the contractive ratios and translations are algebraic numbers, there is a dichotomy: either there are exact overlaps (see the Definition \ref{exact overlap}) or the attractor K satisfies 
$\dim_{H}(K)=\min\{1,\dim_{s}(K)\}$, where $\dim_{s}(K)$ is the similarity dimension which is the unique solution $s$ of $\sum_{i=1}^{m}|r_i^{s}|=1$. We will note that our algorithm is effective when  the attractor is  generated by an  exact overlapping IFS, see Example \ref{inhomgeneous}.
\end{Remark}

Before  proving Theorem \ref{graph-directed}, we introduce the notion of a graph-directed self-similar set.  The following definition is from \cite{MW}, and we use the terminology from their paper. A graph-directed construction in $\mathbb{R}$ consists of the following.
\begin{enumerate}
\item A finite union of bounded closed intervals $\cup_{u=1}^{n}J_u$ such that the  interiors of $J_u$ are pairwise disjoint.
\item A directed graph $G=(V,E)$ with vertex set $V=\{1,\ldots,n\}$ and edge set $E.$ Moreover, we assume that for any $u\in V$ there is some $v\in V$ such that $(u,v)\in E.$
\item For each edge $(u,v)\in E$ there exists a similitude $f_{u,v}(x)=r_{uv}x+a_{uv}$, where $r_{uv}\in (0,1)$ and $a_{uv}\in \mathbb{R}$.  Moreover, for each $u \in V$ the set $\{f_{u,v}(J_v^{\circ}):(u,v)\in E\}$ satisfies the open set condition, i.e.,
there exists $n$ open sets $\{J_{u}^{\circ}:u\in V\}$ such that 
\[\bigcup_{(u,v)\in E}f_{u,v}(J_v^{\circ})\subseteq J_u^{\circ},\] and the elements of $\{f_{u,v}(J_v^{\circ}):(u,v)\in E\}$ are pairwise disjoint, where $J^{\circ}$ is the interior of $J$.
\end{enumerate}
The following result is analogous to the self-similar sets,  the detailed proof can be found in \cite[Theorem 1]{MW}. 
\setcounter{Theorem}{12}
\begin{Theorem}\label{Existence}
 For each graph-directed construction, there exists a unique vector of non-empty compact sets $(C_1,\ldots,C_n)$ such that, for each $u\in V$, $C_{u}=\bigcup_{(u,v)\in E}f_{u,v}(C_v)$.
\end{Theorem}
Let $K^{*}=\cup_{u=1}^{n}C_u$ and call it the graph-directed self-similar set of this construction.
For  each graph-directed construction we can find a weighted adjacency  matrix $A$. This matrix is defined by $A=(r_{u,v})_{(u,v)\in V\times V}$. We let $r_{u,v}=0$ if $(u,v)\notin E$. For each $t\geq 0$ we  can define another weighted adjacency matrix $A^t=(a_{t,u,v})_{(u,v)\in V\times V},$ where $a_{t,u,v}=r_{u,v}^t$.  Let $\Phi(t)$ be the largest nonnegative eigenvalue of $A^t$.
A graph is called  strongly connected if for any two vertices $u,v\in V$, there exists a directed path from $u$ to $v$. A strongly connected component of $G$ is a subgraph $C$ of $G$ such that $C$ is strongly connected, let $SC(G)$ be the set of all the  strongly connected components of $G$.
The following two theorems are the main results of \cite{MW}.
\begin{Theorem}\label{SC}
 For every graph-directed construction such that $G$ is strongly connected, then the Hausdorff dimension of $K^{*}$ is $t_0$, where $t_0$ is the unique solution of the equation $\Phi(t_0)=1$.
\end{Theorem}
 When the  graph-directed construction $G$ is not strongly connected, we can decompose $G$ into several subgraphs which are each strongly connected. Therefore, the following result still holds. 
 \begin{Theorem}\label{NSC}
If the $G$ in our graph-directed construction is not strongly connected, let $t_1=\max\{t_C: \Phi(t_C)=1,\, C\in SC(G)\}$, then $\dim_{H}(K^{*})=t_1$.
\end{Theorem}
By virtue of  Lemma \ref{Partition}, we can define an adjacency matrix $S$ which characterizes the relationship between the different $B_i=A_i\cap K$. Equivalently, we may  define a directed graph $(V, E)$ such that each $B_i$ is associated with a vertex $V_i$ in this graph. Using the
Markov property of the partition, one has an edge from vertex $V_i$ to vertex $V_j$ if there exists some $T_{i,j}\in\{T_1,T_2,\cdots,T_m\}$ such that $B_j \subset T_{i,j}(B_i)$. This allows us to find a similitude, say $g_{i,j}$,  associated to the edge $V_i\to V_j$.  For instance, if $T_{i,j}(x)=T_k(x)$ for some $ 1\leq k\leq m$, then $g_{i,j}(x)=f_k(x)$.  Recall that $A_i=[b_i,b_{i+1}], 1\leq i\leq s$, and note that $\{A_i\}_{i=1}^{s}$ are disjoint except for the endpoints of the intervals, and also forms a Markov partition for the system $\{T_1,T_2,\cdots,T_m\}$. By means of Theorem \ref{Existence}, for the directed graph $(V,E)$ with associated similitudes between vertices, there exists a unique vector of non-empty compact sets $(C_1,\ldots,C_s)$ such that, for each $u\in V$, $C_{u}=\bigcup_{(u,v)\in E}g_{u,v}(C_v)$. We denote $K^{*}=\cup_{i=1}^{s} C_i$ and have following result.
\setcounter{Lemma}{15}
\begin{Lemma}\label{equal}
$K^{*}=K$.
\end{Lemma}
\begin{proof}
For any $x\in K$, by the lazy algorithm, see Definition \ref{Def:28},  we can find the lazy orbit of $x$, i.e. $$\{T_{i_{1}\ldots i_{k}}(x): k\geq 0\},$$ such that 
 for each $k$, $T_{i_{1}\ldots i_{k}}(x)$ falls into some $B_i=A_i \cap K, 1\leq i\leq s$.  As  such in the graph $(V,E)$, we can find  an infinite path 
 $$V_{j_1}\xrightarrow{g_{j_1, j_2}}V_{j_2}\xrightarrow{g_{j_2, j_3}}V_{j_3}\xrightarrow{g_{j_3, j_4}}V_{j_4}\cdots V_{j_n}\xrightarrow{g_{j_n, j_{n+1}}}\cdots .$$
 such that $x= \lim_{n\to \infty}g_{j_1, j_2} \circ\cdots \circ
g_{j_n, j_{n+1}}(0)$.  Each $g_{j_n, j_{n+1}}\in \{f_1,\cdots, f_m\}$. Hence $x\in K^{*}$, which  yields $K\subset K^{*}$. The converse statement 
$ K^{*}\subset K$ is proved similarly.
\end{proof}
Now we prove that $K^{*}$ satisfies the open set condition in terms of the Markov property of the partition. 
\begin{Lemma}\label{OSC}
Let $\{O_i=(b_i,b_{i+1})\}_{i=1}^{s}$. Then  for each $O_i$,  we have that $\bigcup_j g_{i,j} (O_j)\subset O_i$, and the union is disjoint, i.e., the graph-directed self-similar set $K^{*}$ satisfies the open set condition,
where $g_{i,j}(x)=f_k(x)$ for some $1\leq k\leq m$.
\end{Lemma}
\begin{proof}
The union is disjoint as $\{O_j\}$  is. It remains to prove the inclusion. However, this is due to the  Markov property of the partition. More precisely, we have $T_{i,j}(O_i)\supset\bigcup_j (O_j)$, i.e.,  $\bigcup_j T_{i,j}^{-1} (O_j)\subset O_i$, where $T_{i,j}\in \{T_1,T_2,\cdots,T_m\}$. Hence, we  may define some $g_{i,j}(x)=f_k(x), 1\leq k\leq m$, satisfying $\bigcup_j g_{i,j} (O_j)\subset O_i$.
\end{proof}
\begin{proof}[Proof of Theorem \ref{graph-directed}]
The proof follows from Lemma \ref{equal}  and Lemma \ref{OSC}. 
\end{proof}
\subsection{Dimension of $U_{F}$}
 We turn to demonstrating how to find the dimension of the univoque set. In \cite{SKK}, we gave a method which can calculate $\dim_{H}(U_{F})$. However, in \cite{SKK} it was assumed that  $K$ is an interval. In this section,
the self-similar sets are not necessarily  intervals.  We consider the dimension of $U_{F}$ under the same assumptions, i.e. assume the considered  IFS is  exactly overlapping, and the endpoints of $f_i(K),1\leq i\leq m$, are periodic. 
However for simplicity
we assume  that the periodic orbits of the endpoints of  $f_i(K), 1\leq i\leq m,$ never fall into $\cup_{i\neq j}f_i(K)\cap f_j(K)$
as under this assumption 
 each  non-empty $f_i(K)\cap f_j(K)$ is a set of the Markov partition. If the periodic orbits  hit $\cup_{i\neq j}f_i(K)\cap f_j(K)$, then we find the Hausdorff dimension of $U_{F}$ in a similar way. 
 
 Without loss of generality, we assume that there are $t$ non-empty $f_i(K)\cap f_j(K)$. We call them the switch regions, and denote them by $\widehat{B_1}, \widehat{B_2},\cdots, \widehat{B_t}$, more precisely, $\widehat{B_j}=B_{i_j}, 1\leq j\leq t$. These sets are the $i_1, i_2,\cdots i_t$-th sets of the Markov partition.

 Recall the adjacency matrix $S$ corresponding to  the Markov partition. The  corresponding subshift of finite type is
 $$\Sigma=\{(i_k)_{k=1}^{\infty}\in\{1,2,3,\cdots,s\}^{\mathbb{N}}:S_{i_k,i_{k+1}}=1\,\, \mbox{for any}\,\, k\}.$$
We note that the digit $1\leq i\leq s$ is associated with the block $B_i$. 
Now we define another matrix for the calculation of $\dim_{H}(U_F).$

Let $S$ be the adjacency matrix  of the Markov partition,  recall that the indices $i_1,\cdots,i_t$ correspond to the switch regions  $B_{i_1}=\widehat{B_1}, B_{i_2}=\widehat{B_{2}},\cdots, B_{i_t}=\widehat{B_{t}}$. In the matrix $S$,  remove the $i_j$-th row and $i_j$-th column for  $1\leq j\leq t$, and denote this new matrix  by  $S^{'}$.  The corresponding subshift of finite type  is denoted by $\Sigma^{'}$.

We use Example \ref{EX:10} to explain the construction of $S^{'}$. The adjacency matrix is 
 $$
 S=\begin{pmatrix}
  1 & 1 & 1 & 0 \\
 0 & 0 & 0 & 1 \\ 
 0 & 0 & 1 & 1 \\
 1 & 1 & 1 & 1 \\
 \end{pmatrix}
.$$   
Note that there is only one switch region $f_0(K)\cap f_1(K)=f_{0q}(K)=f_{10}(K)$, which is associated with the block $B$. Hence, we cross the second row and second column of $S$, and obtain 
 $$
 S^{'}=\begin{pmatrix}
  1  & 1 & 0 \\
 0  & 1 & 1 \\
 1  & 1 & 1 \\
 \end{pmatrix}
.$$   

Let 
$(j_n)\in \Sigma^{'}$. For any $n\geq 1$, by the definition of Markov partition, there exists an expanding map $T_{j_n, j_{n+1}}$ such that $B_{j_{n+1}}\subset T_{j_n, j_{n+1}}(B_{j_{n}})$. Hence we can define a coding $(d_n)\in \{1,2,\cdots,m\}^{\mathbb{N}}$ in the following way:
$$d_n=i \textrm{ if } T_{j_n, j_{n+1}}=T_i.$$
By the lazy algorithm (Definition \ref{Def:28}), every $(j_n)\in \Sigma^{'}$ is associated with only one $(d_n)\in \{1,2,\cdots,m\}^{\mathbb{N}}$. However, $(d_n)$ may not be the unique coding with respect to $\{f_i\}_{i=1}^{m}$. 
We define the projection of $\Sigma^{'}$ as follows.
\begin{align}\label{Equ2}
\begin{split}
\pi(\Sigma^{'})=\Big\{x=\pi((j_n)): & x=\lim_{n\to \infty}f_{d_1}\circ \cdots\circ f_{d_n}(0)\\
& (d_n) \textrm{ is obtained by the method above} \Big\}.
\end{split}
\end{align}

Now we state the second main result of this paper. 
\setcounter{Theorem}{17}
\begin{Theorem}\label{Univoque set}
Suppose  the IFS of $K$ is  exactly overlapping, and the orbits of  endpoints of each $f_i(K), 1\leq i \leq m,$ are  periodic.  Then  apart from a countable set, $U_{F}$ is a graph-directed self-similar set satisfying the open set condition. 
\end{Theorem}
\setcounter{Lemma}{18}
\begin{Lemma}\label{lem:2.19}
$\dim_{H}(U_{F})\leq \dim_{H}(\pi(\Sigma^{'}))$. 
\end{Lemma}
\begin{proof}
For any $x\in U_{F}$, there exists a unique coding of $x$, say $(i_n)\in \{1, \cdots, m\}^{\mathbb{N}}$,
such that $T_{i_1i_2\cdots i_n}(x)\notin \widehat{B_1}\cup \cdots \cup \widehat{B_t}$ for any $n\geq 0$. By the definition of $\Sigma^{'}$, we have $U_{F}\subset \pi(\Sigma^{'})$. 
\end{proof}
\begin{Lemma}\label{lem:2.20}
There exists a countable set $C$ such that 
$$\pi(\Sigma^{'})\subset U_{F}\cup C.$$
\end{Lemma}
\begin{proof}
Let $x\in\pi(\Sigma^{'}) $. By the definition of $\Sigma^{'}$, there exists $(d_n)\in \{1, \cdots, m\}^{\mathbb{N}}$,
such that $T_{d_1d_2\cdots d_n}(x)\notin \widehat{B_1^{\circ}}\cup \cdots \cup \widehat{B_t^{\circ}}$ for any $n\geq 0$, where $\widehat{B_k^{\circ}}=B_{i_k}^{\circ}=A_{i_k}^{\circ}\cap K=(b_{i_k}, b_{i_k+1})\cap K$. We emphasize here that the orbit we choose is the lazy orbit, see Definition \ref{Def:28}.
If the orbit of $x$ never hits the endpoints of each $B_j, 1\leq j\leq s$, then $x$ is a univoque point, i.e. $x\in U_{F}$. 
If there exists some $n_0$ such that $T_{i_1i_2\cdots i_{n_0}}(x)$ hits some endpoint of $B_j, 1\leq j\leq s$, then we have 
$$x\in \cup_{n=1}^{\infty}\cup_{(i_1\cdots i_n)\in\{1,2,\cdots, m\}^{n}}f_{i_1\cdots i_n}(F),$$
where $F$ is the set of all the endpoints of $\cup_{i=1}^{s}B_i$. 
\end{proof}
Now the remaining  proof of Theorem \ref{Univoque set} is analogous to Theorem \ref{graph-directed}.

We finish this subsection by giving one example. 
The following example was investigated in \cite{NW}. We give a very short calculation. 
\setcounter{Example}{20}
\begin{Example}\label{inhomgeneous}
Let $K$ be the self-similar set of the following IFS:
$$\left\{f_1(x)=\dfrac{x}{3},\, f_2(x)=\dfrac{x}{9}+\dfrac{8}{27},\, f_3(x)=\dfrac{x+2}{3}\right\}.$$
\end{Example}
Let $J=[0,1]$, $f_{1}(J)=\left[0,\dfrac{1}{3}\right],f_{2}(J)=\left[\dfrac{8}{27},\dfrac{11}{27} \right],f_{3}(J)=\left[\dfrac{2}{3},1\right].$
After some calculation, we find that $f_1(K)\cap f_2(K)=f_{1331}(K)\cup f_{1332}(K)\cup f_{1333}(K),$
where $f_{1331}(K)=f_{211}(K), f_{1332}(K)=f_{212}(K), f_{1333}(K)=f_{213}(K), $  i.e. $f_1(K)\cap f_2(K)$ is the union of some exact overlaps.
Hence the IFS is  exactly overlapping, and the endpoints of $f_i(K), 1\leq i\leq 3,$ are periodic. 
We partition $K=A\cup B\cup C\cup D\cup E$, where 

 $A=\left[0, \dfrac{8}{27}\right]\cap K, B=\left[\dfrac{8}{27}, \dfrac{1}{3}\right]\cap K,  C=\left[\dfrac{1}{3}, \dfrac{11}{27}\right]\cap K, D=\left[\dfrac{2}{3}, \dfrac{8}{9}\right]\cap K, E=\left[\dfrac{8}{9}, 1\right]\cap K$.
Then we have $$T_{1}(A)=A\cup B\cup C\cup D, T_{1}(B)=E, $$
and 
$$T_{2}(C)= C\cup D \cup E,\,T_{3}(D)=A\cup B\cup C,\,T_{3}(E)= D\cup E.$$
The weighted adjacency  matrix  for $K$ is 
$$
 A^{t}=\begin{pmatrix}
  3^{-t} &   3^{-t} &  3^{-t} &   3^{-t} & 0 \\
   0 & 0& 0 & 0&    3^{-t}  \\ 
 0 & 0 &    9^{-t} &    9^{-t} &    9^{-t} \\ 
   3^{-t}  &    3^{-t} &   3^{-t}  & 0& 0 \\
 0 & 0 & 0 &   3^{-t} &    3^{-t} \\
 \end{pmatrix}
$$Since the block $B=f_1(K)\cap f_2(K)$ is the union of some exact overlaps, then 
 the weighted adjacency matrix for the univoque set is 
$$
 A^{t}=\begin{pmatrix}
  3^{-t}  &  3^{-t} &   3^{-t} & 0 \\

 0 & 9^{-t} &        9^{-t} &    9^{-t} \\ 
   3^{-t}  &       3^{-t}  & 0& 0 \\
 0 & 0 &   3^{-t} &    3^{-t} \\
 \end{pmatrix}
$$
Therefore the Hausdorff dimension of $K$ is $\dfrac{\log \lambda}{\log 9}=\alpha$, where $\lambda$ is the largest solution of $x^3-6x^2+5x^2-1=0$,
and the dimension of $U_F$ is $\dfrac{\log r}{\log 9}=\gamma$, where $r$ is the largest positive root of 
\[x^5-6x^4+9x^3-8x^2+4x-1=0 \]
Moreover,  since  $S$ and $S^{'}$ are irreducible, we can obtain the property of the Hausdorff measure, i.e.
$$0<\mathcal{H}^{\alpha}(K)<\infty.$$
and
$$0<\mathcal{H}^{\gamma}(U_F)<\infty.$$
\setcounter{Remark}{21}
\begin{Remark}
In \cite{DJKL1},  for this example we  proved that $$\dim_{H}(U_k)=\dim_{H}(U_F)$$ for any finite $k\geq2$, 
where $U_k$ is a subset of $K$ such that every point in $U_k$ has exactly $k$ different codings. 
\end{Remark}

\subsection{Points in $K$ with multiple  codings}
In this subsection we give an application of Theorem \ref{Univoque set}. 
 Let
$$U_{k}:=\{x\in K:x\,\, \mbox{ has exactly $k$ codings}\}, k= 1, 2, \cdots, \aleph_0.$$
A simple analysis enables us to find the Hausdorff dimension of $U_{k}$ for some cases. For simplicity, we consider an example.  The deeper results can be found in \cite{DJKL, DJKL1}. This example contains some key ideas which are useful to analyze $U_{k}$. 
Suppose  that $K$ is the attractor of the following IFS,
$$\left\{f_{1}(x)=\lambda x, f_{2}(x)=\lambda x+2\lambda,  f_{3}(x)=\lambda x+3\lambda-\lambda^2,f_4(x)=\lambda x+1-\lambda\right\}.$$
where $0<\lambda<\dfrac{5-\sqrt{21}}{2}$.

The convex hull of $K$ is $E=[0,1]$. 
$$f_1(E)=[0,\lambda], f_2(E)=[2\lambda, 3\lambda],$$ 
$$ f_3(E)=[ 3\lambda-\lambda^2, 4\lambda-\lambda^2], f_4(E)=[1-\lambda, 1].$$
The first iteration of this IFS is the following figure. 

\begin{figure}[h]\label{figure1}
\centering
\begin{tikzpicture}[scale=10]
\draw(0,0)node[below]{\scriptsize 0}--(.2,0)node[below]{\scriptsize$\lambda$};
\draw(.4,0)node[below]{\scriptsize$2\lambda$}--(0.6,0)node[below]{\scriptsize$3\lambda$};
\draw(0.5,-0.1)node[below]{\scriptsize$3\lambda-\lambda^2$}--(0.7,-0.1)node[below]{\scriptsize$4\lambda-\lambda^2$};
\draw(0.8,0)node[below]{\scriptsize$1-\lambda$}--(1,0)node[below]{\scriptsize$1$};
\node [label={[xshift=1.2cm, yshift=0cm]$f_1(E)$}] {};
\node [label={[xshift=4.8cm, yshift=0cm]$f_2(E)$}] {};
\node [label={[xshift=5.8cm, yshift=-1.2cm]$f_3(E)$}] {};
\node [label={[xshift=8.8cm, yshift=0cm]$f_4(E)$}] {};
\end{tikzpicture}\caption{First iteration of $K$}
\end{figure}
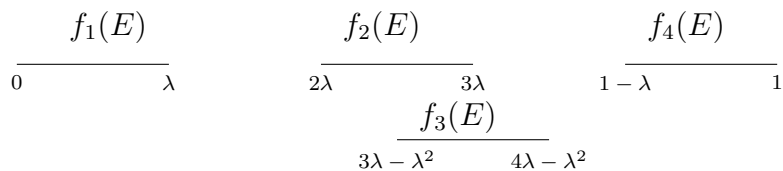

Note that $f_2\circ f_4=f_3\circ f_1$. Hence, we can partition $K$ via 
$$K=A\cup B\cup C\cup D\cup E,$$
where $A=[0,\lambda]\cap K, B=[2\lambda, 3\lambda-\lambda^2]\cap K, C=[3\lambda-\lambda^2,3\lambda]\cap K, D=[3\lambda, 4\lambda-\lambda^2]\cap K, E=[1-\lambda,1]\cap K.$
These blocks have following relations:
$T_1(A)=A\cup B\cup C\cup D\cup E$, $T_2(B)=A\cup B\cup C\cup D$, $T_2(C)=E$, $T_3(D)= B\cup C\cup D\cup E$, $T_4(E)=A\cup B\cup C\cup D\cup E$.
The adjacency matrix is 
$$
 S=\begin{pmatrix}
  1 & 1& 1 & 1 & 1 \\
  1 & 1& 1 & 1 & 0 \\
  0 & 0 & 0 &0&1 \\
0& 1 & 1 & 1&1 \\
  1 & 1& 1 & 1 & 1 \\
 \end{pmatrix}
$$
Since $C=[3\lambda-\lambda^2,3\lambda]\cap K=f_2\circ f_4(K)=f_3\circ f_1(K)=f_2(K)\cap f_3(K)$,
we can define 
$$
 S^{'}=\begin{pmatrix}
  1 & 1 & 1 & 1 \\
  1 & 1& 1 & 0 \\
0& 1  & 1&1 \\
  1 & 1& 1 & 1 \\
 \end{pmatrix}.
$$
We  have following result.
\setcounter{Theorem}{22}
\begin{Theorem}\label{k=1}
Let  $K$ be the attractor of the following IFS,
$$\left\{f_{1}(x)=\lambda x, f_{2}(x)=\lambda x+2\lambda,  f_{3}(x)=\lambda x+3\lambda-\lambda^2,f_4(x)=\lambda x+1-\lambda\right\}.$$
where $0<\lambda<\dfrac{5-\sqrt{21}}{2}$.
Then for any $k\geq 1$,
$\dim_{H}(U_{2^k})=\dim_{H}(U_1)=\dfrac{\log (2+\sqrt{2})}{-\log \lambda}$, where $2+\sqrt{2}$ is the spectral radius of $S^{'}$. Moreover, $U_i=\emptyset, i\neq 2^{k}, k\geq1$, and $U_{\aleph_0}=\emptyset$. 
\end{Theorem}
\setcounter{Corollary}{23}
\begin{Corollary}
For any  $x\in K$, $x$ has $2^k, k\geq 0$, expansions or has uncountable expansions. 
\end{Corollary}
To prove the above results, we need several lemmas.
\setcounter{Lemma}{24}
\begin{Lemma}\label{univoque+}
Suppose $x\in [0,\lambda]\cap K$ has exactly $k$ codings, $k\geq 1$. Then $x+1-\lambda\in [1-\lambda, 1]\cap K$  also has exactly $k$ codings.  Similarly,  if a  point $y\in [1-\lambda,1]\cap K$ has exactly $k$ different codings, then $y-(1-\lambda)$  also has exactly $k$ different codings. 
\end{Lemma}
\begin{proof}
Suppose that  $x\in [0,\lambda]\cap K$ has $k$ codings. Every coding begins with digit $1$, i.e. $(1i_2i_3i_4\cdots)$.
We define a corresponding coding $(4i_2i_3i_4\cdots)$. 
Note that $$x=f_1(\lim_{n\to \infty}f_{i_2}\circ f_{i_3}\circ \cdots \circ f_{i_n}(0)),$$
hence, 
we have 
\begin{eqnarray*}
x+1-\lambda&=&f_4(\lim_{n\to \infty}f_{i_2}\circ f_{i_3}\circ \cdots \circ f_{i_n}(0))\\
\end{eqnarray*}
Therefore, $(4x_2x_3x_4\cdots)$ is a coding of $x+1-\lambda$. 
By the definition of $T_i, 1\leq i\leq 4$, it follows that $T_1(x)=T_4(x+1-\lambda)$. In other words, after the first iteration of the expanding maps, the orbits of $x$ and $x+1-\lambda$ coincide. Thus, if $x\in [0,\lambda]\cap K$ has exactly $k$ codings, then $x+1-\lambda\in [1-\lambda, 1]\cap K$  also has precisely $k$ codings. 
\end{proof}
\begin{Lemma}\label{intersection}
$\dim_{H}(U_2)=\dim_{H}(U_1\cap (U_1+1-\lambda))$.
\end{Lemma}
\begin{proof}
It is sufficient to prove that 
$$U_2=\cup_{n=1}^{\infty}\cup_{(i_1i_2\cdots i_n)\in D } f_{i_1i_2\cdots i_n}\circ f_2( U_1\cap  (U_1+1-\lambda)),$$
where $D$ is the collection of all possible finite blocks appearing in unique codings of $K$. 

For any $x\in U_2 $, there exists a finite block $(i_1i_2\cdots i_n)\in D$ such that $T_{i_1i_2\cdots i_n}(x)$ falls into $f_2(K)\cap f_3(K)$ for the first time, i.e.  $T_{i_1\cdots i_k}(x)\notin f_2(K)\cap f_3(K)$ for $0\leq k\leq n-1$, and $T_{i_1i_2\cdots i_n}(x)\in f_2(K)\cap f_3(K)$.  Here when $k=0$,  $T_{i_0}$ means the identity map. 
 Since $x$ has exactly two codings, it follows that 
$T_2(T_{i_1i_2\cdots i_n}(x)), T_3(T_{i_1i_2\cdots i_n}(x))\in U_1$. By the definition of $T_i,1\leq  i\leq 4$,   we have $T_2(x)=T_3(x)+1-\lambda$.  Hence, $T_3(T_{i_1i_2\cdots i_n}(x))+1-\lambda=T_2(T_{i_1i_2\cdots i_n}(x))$,
which yields $T_2(T_{i_1i_2\cdots i_n}(x))\in U_1\cap (U_1+1-\lambda),$ i.e. $x\in f_{i_1i_2\cdots i_n}\circ f_2( U_1\cap  (U_1+1-\lambda))$. The converse inclusion can be proved similarly. 
\end{proof}

Lemma \ref{univoque+} together with Lemma \ref{intersection} imply the following lemma.
\begin{Lemma}\label{Translation}
$U_1\cap (U_1+1-\lambda)$ is  exactly all the univoque points in $f_4(K)$, i.e.
$$U_1\cap (U_1+1-\lambda)=\{x\in K:  x\,\,\mbox{has a unique coding which has form} \,\, (4x_2x_3\cdots) \},
$$
where $(x_2x_3\cdots)$ is a unique coding. Moreover,
$$\dim_{H}(U_2)=\dim_{H}(U_1\cap (U_1+1-\lambda))=\dim_{H}(U_1).$$
\end{Lemma}
 \begin{Lemma}\label{one side}
 For any $k\geq 1$
$\dim_{H}(U_{2^k})\leq \dim_{H}(U_1)$.
 \end{Lemma}
 \begin{proof}
 The lemma immediately follows from the following inclusion:
 $$U_{2^k}\subset \cup^{\infty}_{n=1}\cup_{(i_1i_2\cdots i_n)\in\{1,2,3,4\}^n}f_{i_1i_2\cdots i_n}(U_1).$$
 \end{proof}
  We  now prove by induction that for any $k\geq 2$, $\dim_{H}(U_{2^k})=\dim_{H}(U_1)$.
 \begin{Lemma}\label{===}
 For any $k\geq 1$, we have
 $f_2\circ f_4 (U_{2^{k}})\subset U_{2^{k+1}}.$
 \end{Lemma}
 \begin{proof}
Given any $x\in f_2\circ f_4 (U_{2^{k}})$, since $f_2\circ f_4=f_3\circ f_1$, and $f_2\circ f_4(K)$ does not intersect with $f_i\circ f_j (K)\in \cup_{(p,q)\notin\{ (2,4),(3,1)\}}\{f_p\circ f_q(K)\}$, it follows that the first two digits of $x$ should be $24$ or $31$.  Therefore, 
 $$x=f_2\circ f_4(y)=f_3\circ f_1(y),$$
 where $y\in U_{2^{k}}$. It is easy to see that  $T_2(x)=f_4(y)$ and $T_3(x)=f_1(y)$ have exactly $2^k$ expansions, respectively. 
 Hence, $x$ has  precisely $2^{k+1}$ expansions.
 \end{proof}
The first statement of Theorem \ref{k=1} now is a straightforward consequence of Lemma \ref{===}, \ref{one side} and \ref{Translation}.
It remains to prove that $U_i=\emptyset, i\neq 2^{k}, k\geq1$. 
\begin{Lemma}\label{odd}
 $U_{2i+1}=\emptyset, i\in  \mathbb{N}$.
\end{Lemma} 
\begin{proof}
We start with proving $U_{3}=\emptyset$. 
If $x\in U_3 \neq \emptyset $, then we can find some  finite block $(i_1i_2\cdots i_n)$ such that $T_{i_1i_2\cdots i_n}(x)$ falls into $f_2(K)\cap f_3(K)$ for the first time. Since $x\in U_3$, it follows that either $T_3(T_{i_1i_2\cdots i_n}(x))$ has a unique coding and $T_2(T_{i_1i_2\cdots i_n}(x))$ has exactly two codings, or $T_2(T_{i_1i_2\cdots i_n}(x))$ has a unique coding and $T_3(T_{i_1i_2\cdots i_n}(x))$ has exactly two codings. We prove  that these two cases are impossible. By Lemma \ref{univoque+}, it is enough to consider the case $T_3(T_{i_1i_2\cdots i_n}(x))$ has a unique coding, and  $T_2(T_{i_1i_2\cdots i_n}(x))$ has two codings. Note that 
$$ T_3((T_{i_1i_2\cdots i_n}(x)))+1-\lambda =T_2(T_{i_1i_2\cdots i_n}(x)).$$
As $T_{i_1i_2\cdots i_n}(x)\in f_2(K)\cap f_3(K)$, we have $T_3((T_{i_1i_2\cdots i_n}(x)))\in [0,\lambda]\cap K$, by Lemma \ref{univoque+}, we know that $$ T_3((T_{i_1i_2\cdots i_n}(x)))+1-\lambda =T_2(T_{i_1i_2\cdots i_n}(x))$$ has a unique coding, which leads to a contradiction. Thus,
$U_3=\emptyset$.  For a general odd number $2k+1$, the proof is similar. If $U_{2k+1}\neq \emptyset$, then there exists a point $x\in U_{2k+1}$ and  a finite sequence $(i_1i_2\cdots i_n)$ such that $T_{i_1i_2\cdots i_n}(x)$ falls into $f_2(K)\cap f_3(K)$ for the first time.  Then, $T_2(T_{i_1i_2\cdots i_n}(x))$ has exactly $a_0$ expansions while $T_3(T_{i_1i_2\cdots i_n}(x))$ has $2k+1-a_0$ expansions, where $1\leq a_0\leq 2k$. 
However, by the same argument as above $T_2(T_{i_1i_2\cdots i_n}(x))$ and $T_3(T_{i_1i_2\cdots i_n}(x))$ have exactly the same number of different expansions, leading to $a_0=2k+1-a_0$ which is impossible. Hence,  $U_{2k+1}=\emptyset$.

 \end{proof}
\begin{Lemma}
$U_{2i}=\emptyset, i\geq 3$, where $2i\neq 2^s$ for any $s\geq 1$. 
\end{Lemma}
\begin{proof} By assumption $2i=2^{\ell}(2m+1)$ for some $\ell\ge 1$ and $m\ge 1$. For each fixed $m$, the proof is done by induction on $\ell$. For $\ell =1$ we have $2i=2(2m+1)$.  If $x\in U_{2(2m+1)}$, then there exists a finite sequence $(j_1j_2\cdots j_n)$ such that $T_{j_1j_2\cdots j_n}(x)$  falls into $f_2(K)\cap f_3(K)$ for the first time. By a similar argument as in Lemma \ref{odd}, the points  $T_2(T_{j_1j_2\cdots j_n}(x))$ and $T_3(T_{j_1j_2\cdots j_n}(x))$ have exactly $i=2m+1$ expansions, which is impossible since $U_{2m+1}=\emptyset$.  Suppose it is true for $2i=2^{\ell}(2m+1)$, i.e. $U_{2^{\ell}(2m+1)}=\emptyset$, and assume $x\in U_{2^{{\ell}+1}(2m+1)}$. Then there exists a finite sequence $(j_1j_2\cdots j_n)$ such that $T_{j_1j_2\cdots j_n}(x)$  falls into $f_2(K)\cap f_3(K)$ for the first time, and as above $T_2(T_{j_1j_2\cdots j_n}(x))$ and $T_3(T_{j_1j_2\cdots j_n}(x))$ have exactly $i=2^{\ell}(2m+1)$ expansions, which is impossible since $U_{2^{\ell}(2m+1)}=\emptyset$. Thus,
$U_{2^{{\ell}+1}(2m+1)}=\emptyset$.
\end{proof}
 Finally, we end the proof of Theorem \ref{k=1} with the following lemma.
\begin{Lemma}\label{countable111}
$U_{\aleph_0}=\emptyset$. 
 \end{Lemma}
 \begin{proof}
 Note that the switch region is $f_2\circ f_4(K)=f_3\circ f_1(K)=f_2(K)\cap f_3(K)$. We decompose the switch region by 
$$f_2(K)\cap f_3(K)=\bigg[\{3\lambda-\lambda^2\}\cup \{3\lambda\}\bigg]\cup \bigg[(f_2(K)\cap f_3(K))\setminus(\{3\lambda-\lambda^2\}\cup \{3\lambda\}) \bigg].$$
Note that $3\lambda-\lambda^2$ has exactly two codings $(241^{\infty})$ and $(31^{\infty})$, and
$3\lambda$ also has exactly two codings $(24^{\infty})$ and $(314^{\infty})$. Assume $x\in K$ has infinitely many codings, we will show $x$ has uncountably many codings. By hypothesis, we can find an orbit of $x$,
denoted  by $\{T^{n}(x)\}^{\infty}_{n=0}$, that hits the switch region infinitely many times. The orbit  $\{T^{n}(x)\}^{\infty}_{n=0}$ cannot hit $3\lambda-\lambda^2$ or $3\lambda$ as $3\lambda-\lambda^2$ and $3\lambda$ have two exactly codings, respectively.  Hence, $\{T^{n}(x)\}^{\infty}_{n=0}$ hits $ \{(f_2(K)\cap f_3(K))\setminus(\{3\lambda-\lambda^2\}\cup \{3\lambda\}) \}$ for infinitely many times. Since $f_2\circ f_4(K)=f_3\circ f_1(K)=f_2(K)\cap f_3(K)$, we can replace the finite block $24$ by $31$ for infinitely many times in the coding of $x$, implying that $x$ has uncountably many codings. So either $x$ has a finite number of expansions or $x$ has uncountably many, hence $U_{\aleph_0}=\emptyset$.
 \end{proof}
 Although for different self-similar sets,
 analyzing the multiple codings of $K$ may vary from each other, the main ideas, i.e. Lemmas \ref{intersection}, \ref{univoque+}, \ref{one side}, \ref{odd}, are very useful to study the multiple codings of $K$. In \cite{DJKL, DJKL1}, some generalizations are considered, and  a new proof of Theorem \ref{k=1} is given. 
\section{Dynamical discription of  $\tilde{U}_{\beta}$}
In  this section we  concentrate on  $\beta$-expansions, and 
give a new dynamical criterion  under which   $\tilde{U}_{F}$ is a subshift of finite type.  For simplicity, we substitute  $U_F$ and $\tilde{U}_{F}$ by  $U_{\beta}$ and $\tilde{U}_{\beta}$, respectively. 
The motivation of this section is to generalize a result of  de Vries  and   Komornik  regarding the calculation of  $\dim_{H}(U_{\beta})$, see \cite[Theorem 1.8]{MK}. We begin with  the definition of the greedy  and quasi-greedy expansions. 
\begin{Definition}
Let $1<\beta<2$, 
the \textbf{greedy map} $G:\mathcal{A}_{\beta}\to \mathcal{A}_{\beta},$ is defined by
\begin{equation*}
G(x)=\left\lbrace\begin{array}{cc}
                 \beta x\mod 1& x\in[0,1)\\

                 \beta x-1& x\in\Big[1, \frac{1}{\beta-1}\Big]
                \end{array}\right.
\end{equation*}
\end{Definition}
For any $n\geq 1$ and $x\in\mathcal{A}_{\beta}$, we define $a_n(x)=\lfloor\beta G^{n-1}(x)\rfloor$, where $\lfloor y\rfloor$ denotes the integer part of $y\in\mathbb{R}$.
We then  have
\begin{eqnarray*}
x&=&\dfrac{a_{1}(x)}{\beta}+\dfrac{G(x)}{\beta}=
 \dfrac{a_{1}(x)}{\beta}+\dfrac{a_2(x)}{\beta^2}+\dfrac{G^2(x)}{\beta^2}\\
&&\vdots\\
&=&\sum_{n=1}^{\infty}\dfrac{a_n(x)}{\beta^n}
\end{eqnarray*}
The sequence  $(a_n)_{n=1}^{\infty}\in\{0,1\}^{\mathbb{N}}$ generated by $G$ is called the \emph{greedy expansion}  for $x$ in base $\beta$ or \emph{greedy coding} for $x$ in base $\beta$. The orbit $\{G^{n}(x)\}_{n=1}^{\infty}$ is called the \emph{greedy orbit} of $x$ in base $\beta$.
In this section, we consider only $\dfrac{1+\sqrt{5}}{2}<\beta<2$ as $U_{\beta}=\{0,(\beta-1)^{-1}\}$ if $1<\beta\leq \dfrac{1+\sqrt{5}}{2}$, see \cite{GS}. 

Similarly, we can define the quasi-greedy orbit  of $1$, denoted it by $\{Q^{n}(1)\}_{n=1}^{\infty}.$
\setcounter{Definition}{1}
\begin{Definition}
If the greedy orbit of $1$ is infinite, we identify the  \textbf{quasi-greedy orbit  of $1$} with the greedy one, i.e.,  
$\{Q^{n}(1)\}_{n=1}^{\infty}=\{G^{n}(1)\}_{n=1}^{\infty}$. Otherwise, let $(a_1a_2\cdots a_n 0^{\infty})$ be the greedy  expansion of $1$, we define the quasi-greedy coding of $1$ by $(a_1a_2\cdots a_n^{-})^{\infty}$, where $a_n^{-}=a_n-1$.  In this case, the quasi-greedy orbit of $1$ is defined by 
$$Q^{i}(1)=(\sigma^{i}( a_1a_2\cdots a_n^{-})^{\infty}))_{\beta}, 1\leq i\leq n$$ 
and $Q^{kn+i}(1)=Q^{i}(1)$ for any $k\geq 0$, 
where $\sigma$ is the left shift, and  $(b_n)_{\beta}:=\sum_{n=1}^{\infty}b_n\beta^{-n}$. For simplicity, throughout this section we denote the quasi-greedy expansion of $1$ by $(\eta_i)_{i=1}^{\infty}$. 
\end{Definition}
\begin{Definition}
We say  \textbf{the quasi-greedy orbit of $1$ hits the point $\beta^{-1}(\beta-1)^{-1}$ ($\beta^{-1}$) for the first time }if there exists a minimal number $k\geq 1$ such that  $Q^{k}(1)=\beta^{-1}(\beta-1)^{-1}$  ($Q^{k}(1)=\beta^{-1}$) and $$Q^{i}(1)\in [0,\beta^{-1})\cup (\beta^{-1}(\beta-1)^{-1}, (\beta-1)^{-1}],\, 1\leq i\leq k-1.$$
\end{Definition}
Similarly, we can define the  quasi-greedy orbit of $1$ hits the interval $(\beta^{-1}, \beta^{-1}(\beta-1)^{-1})$ for the first time, i.e. there exists a smallest $n_0$ such that $Q^{k}(1)\notin (\beta^{-1}, \beta^{-1}(\beta-1)^{-1})$ for any $1\leq k\leq n_0-1$, and $Q^{n_0}(1)\in (\beta^{-1}, \beta^{-1}(\beta-1)^{-1}).$
\setcounter{Lemma}{3}
\begin{Lemma}\label{Remark}
If the quasi-greedy orbit of $1$ hits $\beta^{-1}(\beta-1)^{-1}$ for the first time, then the greedy orbit of $1$ is finite. Moreover there exists some $n_0$ such that the quasi-greedy expansion of $1$ is $$(\eta_i)=(a_1a_2\cdots a_{n_0} \overline{ a_1a_2\cdots a_{n_0})} ^{\infty}$$ for some $(a_1a_2\cdots a_{n_0})$, where $a_{n_0}=1$.
\end{Lemma}
\begin{proof}
Note that if $G^{k}(1)\in[0,\beta^{-1}]\cup [\beta^{-1}(\beta-1)^{-1}, 1] $ for all $1\leq k\leq n$, then  $G^{k}(\bar{1})+G^{k}(1)=(\beta-1)^{-1}$,  where $\bar{1}=(\beta-1)^{-1}-1$. If the quasi-greedy orbit of $1$ hits $\beta^{-1}(\beta-1)^{-1}$ for the first time, then there exists $t_0\geq 1$ such that $G^{t_0}(1)=\beta^{-1}(\beta-1)^{-1}$. Hence, we have that $G^{t_0+1}(1)=\bar{1}$. By symmetry, after $t_0$ step the greedy orbit of $\bar{1}$ will hit $\beta^{-1}$ as $\beta^{-1}+\beta^{-1}(\beta-1)^{-1}=(\beta-1)^{-1}$. Hence the greedy orbit of $1$ is finite.  The second statement follows immediately from the first statement if we take $n_0=t_0+1$.  
\end{proof}
The following theorem characterizes the criteria of the unique expansions, the proof of this result can be found in \cite{MK} or  some references therein. 
\setcounter{Theorem}{4}
\begin{Theorem}\label{Uniquecodings}
 Let $(a_n)_{n=1}^{\infty}$ be a coding of $x\in [0, (\beta-1)^{-1}]$. Then $(a_n)_{n=1}^{\infty}\in \widetilde{U}_{\beta}$ if and only if 
 $$(a_{k+1}a_{k+2}\cdots)<(\eta_n)_{n=1}^{\infty}$$ wherever $a_k=0$, 
  $$\overline{(a_{k+1}a_{k+2}\cdots)}<(\eta_n)_{n=1}^{\infty}$$ wherever $a_k=1$. 
\end{Theorem}
In this theorem $``<"$ means the lexicographic order.  We shall use this symbol when we compare two sequences. There is no risk of confusion in using the
symbol $``<"$ for numbers and for the lexicographic order.
In \cite[Theorem 1.8]{MK},   de Vries and Komornik  proved  the following theorem.
\setcounter{Theorem}{5}
\begin{Theorem}\label{criterion}
$\widetilde{U}_{\beta}$ is a subshift of finite type  if and only if $\beta\in (1,2)\setminus \overline{\mathcal{U}}$, where $\mathcal{U}$ is the set of $\beta$ for which  the $\beta$-expansion of $1$ is unique. 
\end{Theorem}
Equivalentlly we shall prove
\begin{Theorem}\label{Dynamical}
$\widetilde{U}_{\beta}$ is a subshift of finite type if and only if the quasi-greedy orbit of 1 falls into the interval $\left(\dfrac{1}{\beta}, \dfrac{1}{\beta(\beta-1)}\right]$. 
\end{Theorem}
We partition the proof of Theorem \ref{Dynamical} into several lemmas.
\setcounter{Lemma}{7}
\begin{Lemma}\label{QUASI1}
If the quasi-greedy orbit of $1$ hits $\beta^{-1}(\beta-1)^{-1}$ for the first time, then $\tilde{U}_{\beta}$ is a SFT.
\end{Lemma}
Before we prove this lemma, we review one important result, see \cite{MK}. 
\setcounter{Theorem}{8}
\begin{Theorem}\label{quasi-expansion}
Let $(\eta_n)$ be the quasi-greedy expansion of $1$, then $\beta\in \overline{\mathcal{U}}$ if and only if 
$$\overline{\eta_{k}\eta_{k+1}\eta_{k+2}\cdots}<\eta_{1}\eta_{2}\eta_{3}\cdots$$ for any $k\geq 1$. 
\end{Theorem}
By Theorems \ref{criterion} and \ref{quasi-expansion}, to prove Lemma \ref{QUASI1},
  it is enough to prove that there exists $k_0$ such that $\overline{\eta_{k_0}\eta_{k_0+1}\eta_{k_0+2}\cdots}=(\eta_n)$.
  \begin{proof}[Proof of Lemma \ref{QUASI1}]
By Lemma \ref{Remark}, we may assume the quasi-greedy expansion of $1$ is $$(\eta_i)=(a_1a_2\cdots a_{n_0} \overline{ a_1a_2\cdots a_{n_0})} ^{\infty}$$ for some $(a_1a_2\cdots a_{n_0})$. 
From this it follows that  $\sigma^{{n_0}} \overline{(\eta_i)}=(\eta_i)$, and  Lemma \ref{QUASI1} is proved. 
\end{proof}
\setcounter{Lemma}{9}
\begin{Lemma}\label{HITINTERIOR}
If  there exists $k_0\geq 1$ such that 
$$\overline{\eta_{k_0+1}\eta_{k_0+2}\eta_{k_0+3}\cdots}>\eta_{1}\eta_{2}\eta_{3}\cdots,$$ then the quasi-greedy orbit of $1$ hits the interior of the switch region. \end{Lemma}
\begin{proof}
Suppose $$\eta_{k_0+1}\eta_{k_0+2}\eta_{k_0+3}\cdots< \overline{\eta_{1}\eta_{2}\eta_{3}\cdots}$$ for some $k_0\geq 1$,
 and let  $(\alpha_i)$ be  the quasi-greedy expansion of $\bar{1}$.
 Since the quasi-greedy expansion is the largest infinite sequence in the sense of  lexicographical ordering, we have 
 $$\overline{\eta_{1}\eta_{2}\eta_{3}\cdots}\leq (\alpha_i).$$   Thus, $$\eta_{k_0+1}\eta_{k_0+2}\eta_{k_0+3}\cdots<(\alpha_i). $$
By  monotonicity of the quasi-greedy expansion \cite{MK}, it follows  that $$(\eta_{k_0+1}\eta_{k_0+2}\eta_{k_0+3}\cdots)_{\beta}:=\sum_{j=1}^{\infty}\eta_{j+k_0}\beta^{-j}<\bar{1}.$$ Hence, the quasi-greedy orbit of $1$ must fall into $(\beta^{-1},\beta^{-1}(\beta-1)^{-1})$.
\end{proof}
\begin{Lemma}\label{QUASI22}
If the quasi-greedy orbit of $1$ hits $\beta^{-1}$ for the first time, then we have that $$\overline{\eta_{k+1}\eta_{k+2}\eta_{k+3}\cdots}<\eta_{1}\eta_{2}\eta_{3}\cdots,$$ for any $k\geq 1$. Consequently, $\tilde{U}_{\beta}$ is not a subshift of finite type. 
\end{Lemma}
\begin{proof}
Suppose the quasi-greedy orbit of $1$ hits $\beta^{-1}$ for the first time, but 
$$\eta_{k_0+1}\eta_{k_0+2}\eta_{k_0+3}\cdots\leq \overline{\eta_{1}\eta_{2}\eta_{3}\cdots}$$ for some $k_0\geq 1$.
By Lemma \ref{HITINTERIOR}, 
we have $$\eta_{k_0+1}\eta_{k_0+2}\eta_{k_0+3}\cdots\geq \overline{\eta_{1}\eta_{2}\eta_{3}\cdots}$$ implying
$$\eta_{k_0+1}\eta_{k_0+2}\eta_{k_0+3}\cdots= \overline{\eta_{1}\eta_{2}\eta_{3}\cdots}.$$
Therefore
there exists $k_0\geq 1$ such that $\sigma^{k_0}(\overline{\eta_i})=(\eta_i)$.
By the assumption of the lemma, we know that  the  greedy orbit of $1$ is finite. Hence, the quasi-greedy expansion of $1$ has the form
$(\eta_i)=(a_1a_2\cdots a_n^{-})^{\infty}$ for some $(a_1a_2\cdots a_n^{-})$, where $a_n=1$.
Since $(\eta_i)=(a_1a_2\cdots a_n^{-})^{\infty}$, it follows that the quasi-greedy orbit of $1$ never hits the point $(\beta)^{-1}(\beta-1)^{-1}$, and subsequently never hits the point $\bar{1}=(2-\beta)(\beta-1)^{-1}$. 
However, we have $\sigma^{k_0}(\eta_i )= (\overline{(\eta_i)})$, which yields that the quasi-greedy orbit of $1$ hits $\bar{1}$, leading a contradiction. 
\end{proof}
\begin{proof}[Proof of Theorem \ref{Dynamical}]
If the quasi-greedy orbit of 1 hits the interior of switch region or  $(\beta)^{-1}(\beta-1)^{-1}$ for the first time, then $\tilde{U}_{\beta}$ is a subshift of finite type in terms of  \cite[Theorem 2.4]{SKK} and Lemma \ref{QUASI1}.
Conversely, if $\tilde{U}_{\beta}$ is a subshift of finite type, then we have $\beta\notin \overline{\mathcal{U}}$, which implies that the quasi-greedy expansion of $1$ hits the switch region eventually. By Lemma \ref{QUASI22}, it follows that 
the quasi-greedy orbit of 1 hits the interior of switch region or  $(\beta)^{-1}(\beta-1)^{-1}$ for the first time. 
\end{proof}

\section{Univoque set for $\beta$-expansions}
    Let $1<\beta<2$. In this section, we give another application of our result to $\beta$-expansions. For simplicity, we still use the definitions and notation defined in the last section. 
The main goal  of this section  is to calculate $\dim_{H}(U_{\beta})$ when $\tilde{U}_{\beta}$ is not a subshift of finite type, in general this is a hard task. However, in this  scenario the technique of periodic  orbits 
allows us to find $\dim_{H}(U_{\beta})$ in many cases. 
By Theorem 2.4 in \cite{SKK} we have that $\tilde{U}_{\beta}$ is a subshift of finite type if the orbit of $1$ falls into $(\beta^{-1}, \beta^{-1}(\beta-1)^{-1})$ and that  $\dim_{H}(U_{\beta})$ can be calculated in terms of Mauldin and Williams' work \cite{MW}. Hence we suppose in this section that the greedy orbit of $1$ is eventually periodic and that it never hits $(\beta^{-1}, \beta^{-1}(\beta-1)^{-1})$.  

One can express the univoque set $U_{\beta}$ as
$$U_{\beta}=\left\{x\in[0, (\beta-1)^{-1}]: G^{n}(x)\notin [\beta^{-1}, \beta^{-1}(\beta-1)^{-1}]\,\, \mbox{for any }\,\,n\geq 0\right \}.$$

Now we   partition $[0,(\beta-1)^{-1}]$ via the   eventually periodic greedy orbits of $1$ and $(2-\beta)(\beta-1)^{-1}$.
Let $\bar{1}=(2-\beta)(\beta-1)^{-1}$ be the reflection of 1.   Since the greedy orbit of $1$ takes only finitely many values, by symmetry it follows that the greedy orbit of $\bar{1}$ also takes only finitely many values, see the first statement in the proof of Lemma \ref{Remark}. Without loss of generality, we may assume that the values of the greedy orbits of $1$ and $\bar{1}$  are $c_1<c_2<\cdots<c_p$. If $\beta^{-1}$, $\beta^{-1}(\beta-1)^{-1}$, $0$,  $1$  and $(\beta-1)^{-1}$ are not in $\{c_1,c_2,\cdots,c_p\}$, then we add these  points to this set. For simplicity, we  still assume that our partition points of $[0,(\beta-1)^{-1}]$ are $\{c_1,c_2,\cdots,c_p\}$. 
Now we have that 
$$[0,(\beta-1)^{-1}]=\bigcup_{i=1}^{p-1}[c_i, c_{i+1}].$$
It is easy to see that the image of each interval of the partition is the union of some elements of the Markov partition, see the following example.
\begin{Example}\label{Multi}
Let $\beta$ be the largest positive root of $x^{3}=x^{2}+x+1$. 
The  greedy orbit of $1$ takes three vaules, i.e. $G(1)=\beta-1, G^2(1)=\beta^{-1}, G^n(1)=0, n\geq  3$. Similarly,  $G(\bar{1})=(\beta-1)^{-1}\beta^{-2}$,  $G^{2}(\bar{1})=(\beta-1)^{-1}\beta^{-1}$, $G^{3}(\bar{1})=\bar{1}$.

Let $A=\left[0, \dfrac{2-\beta}{\beta-1}\right]$, $B=\left[ \dfrac{2-\beta}{\beta-1},\dfrac{1}{(\beta-1)\beta^2} \right]$, 
 $C=\left[ \dfrac{1}{(\beta-1)\beta^2},\dfrac{1}{\beta} \right]$,  $D=\left[\beta^{-1}, \dfrac{1}{(\beta-1)\beta}\right]$, $E=\left[\dfrac{1}{(\beta-1)\beta}, \beta-1\right]$,
$F=\left[\beta-1,1\right]$, $G=[1, (\beta-1)^{-1}]$.
Hence we give a Markov partition of $[0,1)$ and they have following relations:

$T_0(A)=A\cup B$, $T_0(B)=C\cup D$, $T_0(C)=E\cup F$, $T_1(D)=A$, $T_1(E)=B\cup C$, $T_1(F)=D\cup E$, $T_1(G)=F\cup G.$
The corresponding adjacency matrix is given by  
$$
 S=\begin{pmatrix}
  1 & 1 & 0 & 0& 0 & 0 & 0\\
   0 & 0& 1 & 1& 0 & 0 & 0\\ 
 0 & 0 & 0 & 0& 1 & 1& 0\\ 
 1 & 0 & 0 & 0& 0 & 0& 0\\
 0 & 1 & 1 & 0& 0 & 0& 0\\
 0 & 0 & 0 & 1& 1 & 0& 0\\
  0 & 0 & 0 & 0& 0 & 1& 1\\
 \end{pmatrix}.
$$
\end{Example}
Here we should emphasize that the algorithm in the example is not the greedy map for all the points in $[0, (\beta-1)^{-1}]$. For example, for the point $\beta^{-1}$, we can implement $T_0$ (as $\beta^{-1}$ is the right endpoint of $[\beta^{-2}(\beta-1)^{-1} ,\beta^{-1}]$) or $T_1$ (as $\beta^{-1}$ is the left endpoint of $[ \beta^{-1}, \beta^{-1}(\beta-1)^{-1}]$). But for the sets of partition, we do use the greedy algorithm, i.e. for any $[c_i, c_{i+1}]$, we find the largest $k$ such that $T_k$ works on $[c_i, c_{i+1}]$. 

Recall the Markov partition of  $[0, (\beta-1)^{-1}]= \bigcup_{i=1}^{p-1}[c_i, c_{i+1}]$, we identify each intervals $[c_i, c_{i+1}]$ with a block $A_i$. In this section, we use the set of blocks $\{A_i\}_{i=1}^{p-1}$  representing the intervals $\{[c_i, c_{i+1}]\}_{i=1}^{p-1}$. 
Let $S$ be the corresponding adjacency matrix of the Markov partition, and $\Sigma$ the associated subshift of finite type, i.e.
$$ \Sigma=\{({i_n})\in\{1,2,3\cdots,p-1\}^{\mathbb{N}}:S_{{i_n}, {i_{n+1}}}=1\}.$$
Note  that we are assuming that  the greedy orbit of $1$ never hits $(\beta^{-1},\beta^{-1}(\beta-1)^{-1})$. 
 Hence, by the definition of  Markov partition, the switch region is one of the sets of this partition.  We denote its associated block by $A_i$. 

Let $S$ be the underlying  adjacency matrix, and  the associated block of the switch region be $A_i$. We remove the $i$-th row and $i$-th column of $S$, and denote this new matrix by $S^{'}$.  
For  Example \ref{Multi} we mentioned above, 
$$
 S^{'}=\begin{pmatrix}
  1 & 1 & 0 & 0 & 0 & 0\\
   0 & 0& 1 & 0 & 0 & 0\\ 
 0 & 0 & 0 & 1 & 1 & 0\\ 
 0 & 1 & 1 & 0 & 0 & 0\\
 0 & 0 & 0 & 1 & 0 & 0\\
 0 & 0 & 0 & 0 &1 &1\\
 \end{pmatrix}
$$
Using the matrix $S^{'}$, we can define a subshift of finite type $\Sigma^{'},$
i.e. 
$$ \Sigma^{'}=\{({i_n})\in\{1,2,3\cdots,p-2\}^{\mathbb{N}}:S^{'}_{{i_n}, {i_{n+1}}}=1\}.$$
Let 
$(j_n)\in \Sigma^{'}$. For any $n\geq 1$, by the definition of Markov partition, there exists an expanding map $T_{j_n}=T_0 \mbox{ or } T_1$ such that $A_{j_{n+1}}\subset T_{j_n}(A_{j_{n}})$. Hence we can define a coding $(a_n)\in \{0,1\}^{\mathbb{N}}$ in the following way:
$$a_n=i \textrm{ if } T_{j_n}=T_i, i=0 \mbox{ or }1.$$
Define the projection of $\Sigma^{'}$. 
$$
\pi(\Sigma^{'})=\Big\{x=\pi((j_n)):  x=\sum_{n=1}^{\infty}\dfrac{a_n}{\beta^n},  
 (a_n) \textrm{ is obtained by the method above} \Big\}.
$$
Note that after we delete the switch region in the Markov partition $\{A_i\}_{i=1}^{p-1}$, $\beta^{-1}$ and 
$\beta^{-1}(\beta-1)^{-1}$ are still in $\pi(\Sigma^{'})$. We give a simple proof for this fact.  As we delete the switch region, i.e. correspondingly, we delete the block $A_i=[\beta^{-1}, \beta^{-1}(\beta-1)^{-1}]$ in the partition $\cup_{k=1}^{p-1}A_k$,  the point $\beta^{-1}$ only belongs to the block $A_{i-1}$, that is,  $\beta^{-1}$ is the right endpoint of $A_{i-1}$. Hence  we can only  implement $T_0$ on $\beta^{-1}$. Then the orbit of $1$ goes to the point $1$,
By the  assumption,  the greedy orbit of $1$ does not hit $(\beta^{-1}, \beta^{-1}(\beta-1)^{-1})$. Hence, the orbit of $1$ is always in $[0,\beta^{-1}]\cup [\beta^{-1}(\beta-1)^{-1},1]=\cup_{k=1}^{i-1}A_k\cup \cup_{k=i+1}^{p-1}A_k$, which implies that $\beta^{-1}\in \pi(\Sigma^{'})$. Analogous discussion is still correct for $\beta^{-1}(\beta-1)^{-1}$. 

Now we  state  the main result of  this section.
\setcounter{Theorem}{1}
\begin{Theorem}\label{Maintheorem}
Let  $1<\beta<2$. If the quasi-greedy orbit of $1$ does not hit $(\beta^{-1}, \beta^{-1}(\beta-1)^{-1}]$, and  the greedy expansion of $1$ in base $\beta$ is eventually periodic,  then we have $\dim_{H}(U_{\beta})=\dfrac{\log r}{\log \beta}$, where $r$ is the spectral radius of matrix $S^{'}$. 
\end{Theorem}
We give an outline for the proof of this theorem. 
The main ideas are analogous to  the proofs of Lemmas \ref{lem:2.19} and \ref{lem:2.20}. 
We can show that there exists a countable set $C$ such that
$U_{\beta}\subset \pi(\Sigma^{'})\subset U_{\beta}\cup C$, i.e. $\pi(\Sigma^{'})=U_{\beta}\cup C$, where $C=\cup_{n=0}^{\infty}\cup_{(i_1\cdots i_n)\in \{0,1\}^{n}}f_{i_1\cdots i_n}\{\beta^{-1}, \beta^{-1}(\beta-1)^{-1}\}$.  Then $\pi(\Sigma^{'})$ is a graph-directed self-similar set satisfying the open set condition, see  Lemma \ref{OSC}. Therefore, the stated result is proved. 
\setcounter{Remark}{2}
\begin{Remark}
We explain here why we generalize  de Vries and Komornik's result regarding  the calculation of  $\dim_{H}(U_{\beta})$. In the last section, we proved a necessary and sufficient condition  which can determine when  $\tilde{U}_{\beta}$ is a subshift of finite type. If the quasi-greedy orbit of $1$ hits $(\beta^{-1},\beta^{-1}(\beta-1)^{-1})$, then $\tilde{U}_{\beta}$ is a subshift of finite type. For this case we can implement the idea of Theorem 2.4 from  \cite{SKK}, and find $\dim_{H}(U_{\beta})$ using the machinery of  Mauldin and Williams \cite{MW}, see the details in \cite{SKK}. If  the quasi-greedy orbit of $1$ hits $\beta^{-1}(\beta-1)^{-1}$ for the first time, then $\tilde{U}_{\beta}$ is  also a  subshift of finite type. For this case, the greedy expansion of $1$ is finite (Lemma \ref{Remark}), and we can make use of Theorem \ref{Maintheorem}  to calculate $\dim_{H}(U_{\beta})$. 
However, Theorem \ref{Maintheorem} enables us to calculate some cases such that $\tilde{U}_{\beta}$ is not a subshift of finite type,  see some examples below.  As such we generalize  the result of    de Vries and Komornik. Moreover, if the self-similar set $K$ is an interval, then the IFS is not necessarily   exactly overlapping.  In this case the univoque set  can  be identified with an open dynamical system.  Hence, similar results as Theorem \ref{Maintheorem} can be found. 
\end{Remark}

Motivated by Theorem \ref{Maintheorem}, we prove that $\tilde{U}_{\beta}$ is a sofic shift if and only if $\tilde{U}_{\beta}$ is  a  subshift of finite type. This result is essentially proved by de Vries and Komornik \cite[Theorem 1.8]{MK}. To do so, we introduce some definitions and results  for sofic shift. We adopt the definitions from  Chapter 3 of \cite{LM}. 
\setcounter{Definition}{3}
\begin{Definition}
Let $X$ be the full shift over alphabet $\mathcal{A}$. A subset $Y$ of $X$ is called a \textbf{subshift} if $Y=X_{\mathcal{F}}$ for some collection $\mathcal{F}$ of forbidden blocks over $\mathcal{A}$, i.e. any sequence in  $Y$ does not contain  blocks from $\mathcal{F}$. 
If the collection $\mathcal{F}$ is finite, then $Y$ is called a \textbf{subshift of finite type}. 
\end{Definition}
\begin{Definition}\label{SoficDefinition}
Let $(\Sigma_{S}, \sigma)$ be a subshift of $\{0,1\}^{\mathbb{N}}$.  We say $(\Sigma_{S}, \sigma)$ is a \textbf{sofic shift} if it is a factor of  some subshift of finite type $(\Sigma, \sigma)$, i.e. there exists a continuous map  $\phi: \Sigma\to \Sigma_{S}$ such that  $\phi$ is onto
and 
$$\sigma\circ \phi=\phi\circ \sigma.$$
The map $\phi$ is called a \textbf{factor map}.  If  $(\Sigma_{S}, \sigma)$  is a sofic shift but not a subshift of finite type, then we call  $(\Sigma_{S}, \sigma)$  a \textbf{strictly sofic shift}. 
\end{Definition}
In \cite{LM} Lind and Marcus give another equivalent definition of  a sofic shift.  For the greedy $\beta$-expansions, 
 some related results should be mentioned. The following result was proved by Akiyama and Scheicher \cite{ASK}. 
\setcounter{Theorem}{5}
\begin{Theorem}
Given any $1<\beta<2$. Then  the set of all the greedy expansions in base $\beta$ is a sofic shift if and only if the greedy expansion of $1$ is eventually periodic.  Moreover, 
the set of all the greedy expansions in base $\beta$ is a subshift of finite type  if and only if the greedy expansion of $1$ is finite.
\end{Theorem}
However, for $\tilde{U}_{\beta}$, the following result is a little surprising. Our main result is that $\tilde{U}_{\beta}$ is a sofic shift if and only if $\tilde{U}_{\beta}$ is a subshift of finite type. Denote by $\mathcal{U}$  the set of $\beta$ for which the $\beta$-expansion of 1 is unique.
\begin{Theorem}\label{Sofic}
Let $1<\beta<2$. 
The following statements are equivalent. 
\begin{itemize}
\item [1)] $\beta\in (1,2)\setminus \overline{\mathcal{U}}$. 
\item [2)] $\tilde{U}_{\beta}$ is a subshift. 
\item [3)] $\tilde{U}_{\beta}$ is a subshift of finite type.
\item [4)]  $\tilde{U}_{\beta}$ is a closed set. 
\item [5)]  $\tilde{U}_{\beta}$ is a sofic shift.
\end{itemize}
\end{Theorem}
\begin{proof}
For the equivalence of the first four statements, the detailed proofs can be found in \cite{MK}. 
It is well known that  every subshift of finite type is a sofic shift, \cite[Theorem 3.1.5]{LM}. On the other hand, if $\tilde{U}_{\beta}$ is a sofic shift, then by the Definition \ref{SoficDefinition}, $\tilde{U}_{\beta}$ is a factor of a subshift of finite type $\Sigma^{''}$, i.e. there exists a continuous map $\phi:\Sigma^{''}\to \tilde{U}_{\beta}$ such that 
$$\phi(\Sigma^{''})=\tilde{U}_{\beta}.$$
Since $\Sigma^{''}$ is compact and $\phi$ is continuous, it follows that $\tilde{U}_{\beta}$ is compact and hence closed. Therefore, by the fourth  and third equivalent conditions of Theorem \ref{Sofic}, $\tilde{U}_{\beta}$ is a subshift of finite type. 
\end{proof}
\setcounter{Corollary}{7}
\begin{Corollary}
For any $\dfrac{1+\sqrt{5}}{2}<\beta<2$, 
 $\tilde{U}_{\beta}$ cannot be a  strictly sofic shift. 
 \end{Corollary}
We finish this section by giving some examples and remarks. 
\setcounter{Example}{8}
\begin{Example}
Let $\beta\approx 1.8393$ be the appropriate root of $x^3=x^2+x+1$. Then  $$\dim_{H}(U_{\beta})=\dfrac{\log G}{\log \beta},$$ where $G$ is the golden mean. 
 
In view of Theorem \ref{Maintheorem}, it remains to calculate the spectral radius of $S^{'}$. 
Here, for simplicity we only consider the univoque points in $[0,1]$ as the greedy orbits eventually fall  into $[0,1]$. All the following examples use this idea. 
From Example \ref{Multi}, we know that 
$$
 S^{'}=\begin{pmatrix}
  1 & 1 & 0 & 0& 0 \\
   0 & 0& 1 & 0& 0 \\ 
 0 & 0 & 0 & 1& 1 \\ 
 0 & 1& 1 & 0& 0 \\
 0 & 0 & 0 & 1& 0 \\
 \end{pmatrix}
,$$
therefore the spectral radius of this matrix is $G$. 
\end{Example}
\setcounter{Remark}{9}
\begin{Remark}
In \cite{GS}, Glendenning and Sidorov stated (without proof) that for any multinacci number $c_n$, i.e. the largest positive root of $x^n=\sum_{i=1}^{n-1}x^i, n\geq 3$,   $\dim_{H}(U_{\beta})=\dfrac{\log c_{n-1}}{\log c_n}$. Theorem \ref{Maintheorem} gives a proof of this result.  For this class, the quasi-greedy orbit of $1$ hits $\beta^{-1}$ for the first time, which means that $\tilde{U}_{\beta}$ is not a subshift of finite type. 
\end{Remark}
\setcounter{Example}{10}
\begin{Example}
Let $\beta\approx 1.8668$ be the Pisot number satisfying $ x^4-2x^3+x-1=0$, then $\dim_{H}(U_{\beta})=\dfrac{\log G}{\log\beta}$, where $G$ is the golden mean.
$$
 S^{'}=\begin{pmatrix}
  1 & 1 & 0 & 0& 0 \\
   0 & 0& 1 & 0& 0 \\ 
 0 & 0 & 0 & 1& 1 \\ 
 0 & 1 & 1 & 0& 0 \\
 0 & 0 & 0 & 1& 0 \\
 \end{pmatrix}
$$
For this example the quasi-greedy orbit of $1$ hits $\beta^{-1}(\beta-1)^{-1}$ for the first time, and $\tilde{U}_{\beta}$ is a subshift of finite type. 
\end{Example}
\begin{Example}
Let $\beta$ be the  largest positive root of $x^4-x^3-2x^2+1=0$, then $\dim_{H}(U_{\beta})=\dfrac{\log r}{\log \beta}\approx\dfrac{\log 1.7693}{\log \beta}$, where $r$ is the real root of $x^3-2x-2=0$. 
$$
 S^{'}=\begin{pmatrix}
  1 & 1 & 0 & 0 & 0&0 & 0\\
 0 & 0 & 1 & 0 & 0&0 & 0\\ 
0& 0 & 0 & 1 & 1&0 & 0\\
  0 & 0 & 0 & 0 & 0&1 & 1\\
0& 1 & 1 & 0 & 0&0 & 0\\
0& 0 & 0 & 1 & 1&0 & 0\\
 0 & 0 & 0 & 0 & 0&1 & 0\\ 
 \end{pmatrix}
.$$
The spectral radius  of this matrix is the  real root of $x^3-2x-2=0$. The solution is about $1.7693$.  For this example, $1$ has a unique coding in base $\beta$, and this $\beta$ is a Pisot number.  In this case, $\tilde{U}_{\beta}$ is not a subshift of finite type. 
\end{Example}
\setcounter{Remark}{12}
\begin{Remark}
Recently, Komornik, Kong and Li \cite{KKLL} gave a uniform formula of $\dim_{H}(U_{\beta})$, i.e., $ \dim_{H}(U_{\beta})=\dfrac{h(\tilde{U}_{\beta})}{\log \beta}$, where $ h(\tilde{U}_{\beta})$ is the entropy of $\tilde{U}_{\beta}$. In fact, Glendinning and Sidorov \cite{GS} mentioned this result in their paper (without proof). It is very difficult to calculate the  entropy of $\tilde{U}_{\beta}$ for every $\beta>1$.  However, our main results, namely \cite[Theorem 2.4]{SKK}, Theorem \ref{Maintheorem}  and Theorem \ref{Dynamical}  allow us to calculate $\dim_{H}(U_{\beta})$ for almost all $\beta>1$. Similar result is also obtained in \cite{KL} by a different method.
\end{Remark}
\section{Doubling map with holes}
In this section, we consider the doubling map with holes. Our main motivation is to answer partially one problem posed in the PhD thesis of   Alcaraz Barrera \cite{Barrera}. 
We start with some definitions.
Let $T(x)=2 x \mod 1$ be the doubling map defined on $[0,1)$. Given any $(a,b)\subset [0,1)$, define
$$J(a,b):=\{x\in [0,1):T^{n}(x)\notin (a,b), \forall\, n\geq 0\}.$$
$$\tilde{J}(a,b):=\{(a_n)\in\{0,1\}^{\mathbb{N}}: (a_n) \textrm{ is the binary expansions of } x \in J(a,b)\}.$$
We refer the interval $(a,b)$ as the hole. 
  Alcaraz Barrera posed the following question in his PhD thesis.
\begin{Question}\label{QUE}
For which hole $(a,b)\subset \left(0,\dfrac{1}{2}\right)$ with the property that the orbits of $a$ and $b$ do not hit $(a,b)$, is $J(a,b)$  measure theoretically  isomorphic to a subshift of finite type?
\end{Question}
\setcounter{Remark}{1}
\begin{Remark}
If  the orbits of $a$ and $b$  hit $(a,b)$, then  $\tilde{J}(a,b)$ is a subshift of  finite type. This fact was proved in \cite[Proposition 4.1]{STS}. That is why the extra assumption  is added on the orbits of $a$ and $b.$
\end{Remark}
Our  partial answer to Question \ref{QUE} is the following theorem.
\setcounter{Theorem}{2}
\begin{Theorem}\label{subshift}
For any hole $(a,b)\subset \left(0,\dfrac{1}{2}\right)$, if the orbits of $a$ and $b$ are eventually periodic, and the associated  adjacency matrix of $J(a,b) $, i.e. $S^{'},$ is irreducible, then $J(a,b)$ is measure theoretically  isomorphic to a subshift of finite type.
\end{Theorem}
\setcounter{Remark}{3}
\begin{Remark}
The underlying measure and the matrix $S^{'}$ in the theorem will be given later in this section.   The reason why we need the irreducibility condition is to define the Parry measure \cite{Par64}. 
An analogous result holds for several holes. It is easy to see that any point  $x\in [0,1)$ can be approximated by a sequence $(x_n)_{n=1}^{\infty}\in [0,1)$ such that each $x_n$ has an eventually periodic coding. With this observation,  for many pairs $(x_n, y_n)$, if $x_n$ and $y_n$ have eventually periodic codings, then by Theorem \ref{subshift},
the event that  $J(x_n,y_n)$ is   measure theoretically isomorphic to a subshift of finite type happens for many cases. Calculating the Hausdorff dimension of $J(a,b)$ is indeed analogous to Theorem \ref{Maintheorem}. For this problem irreducibility condition  is not necessary.  We do not give the detailed proof for this case. 
\end{Remark}
The proof of Theorem \ref{subshift} follows the same idea used in section 2. 
Since the  orbits of $a$ and $b$ are eventually periodic, we can partition $[0,1)$ by the orbits of $a$ and $b$. 
Denote the values of  the orbits of $a$ and $b$ by $D=\{d_1,d_2,\cdots, d_p\}$. If $0, 2^{-1},a,b$ and $1$ are not in this set, we put these  points in  $D$. Without loss of generality, we still use the set $D=\{0=d_1,d_2,\cdots,d_p=1\}$. These points give a partition of $[0,1)$, i.e., 
$$[0,1)=\cup_{i=1}^{p-1}[d_i,d_{i+1}).$$
Note that $1$ is not a point of the partition.  We still use the notation above as it will not effect our main result. 
It is not difficult to see that each $[d_i,d_{i+1})$ is mapped into an interval which is the union of some sets of the partition, see the following example.
\setcounter{Example}{4}
\begin{Example}\label{example}
Let $a=(01010)^{\infty}_2 $ and $b=(10010)^{\infty}_2$, then the orbit of a is $$\left\{a=\dfrac{10}{31}, T(a)=\dfrac{20}{31}, T^2(a)=\dfrac{9}{31},T^3(a)=\dfrac{18}{31},T^4(a)=\dfrac{5}{31} \right\}$$
the orbit of $b$ is 
$$\left\{b=\dfrac{18}{31},T(b)=\dfrac{5}{31},T^2(b)=\dfrac{10}{31},T^3(b)=\dfrac{20}{31},T^4(b)=\dfrac{9}{31} \right\},$$
We  partition $[0,1)$ as follows, 
\begin{eqnarray*}
\left[0,1\right)&=&\left[0, \dfrac{5}{31}\right)\cup \left[\dfrac{5}{31}, \dfrac{9}{31}\right)\cup \left[\dfrac{9}{31}, \dfrac{10}{31}\right)\cup \left[\dfrac{10}{31}, \dfrac{1}{2}\right)\cup \left[\dfrac{1}{2}, \dfrac{18}{31}\right)\cup \left[\dfrac{18}{31}, \dfrac{20}{31}\right)\cup \left[\dfrac{20}{31}, 1\right)\\
&=&A\cup B\cup C\cup D\cup E\cup F\cup G
\end{eqnarray*}
This is  a Markov partition of $[0,1)$ as we have

$T(A)=A\cup B\cup C,T(B)=D\cup E,T(C)=E,T(D)=G,T(E)=A,T(F)=B$ and $T(G)=C\cup D\cup E\cup F\cup G$.
Subsequently we can define an adjacency matrix
$$
 S=\begin{pmatrix}
  1 & 1 & 1 & 0& 0 & 0& 0\\
  0 & 0 & 0 & 1& 1 & 0& 0\\  
0 & 0 & 0 & 0& 1 & 0& 0\\  
0 & 0 & 0 & 0& 0 & 0& 1\\
  1 & 0 & 0 & 0& 0 & 0& 0\\
  0 & 1 & 0 & 0& 0 & 0& 0\\
  0& 0 & 1 & 1& 1 & 1& 1\\
 \end{pmatrix}.
$$
\end{Example}
Recall the Markov partition of  $[0,1)= \bigcup_{i=1}^{p-1}[d_i, d_{i+1})$, we identify each intervals $[d_i, d_{i+1})$ with a block $B_i$. In what follows, we use the set of blocks $\{B_i\}_{i=1}^{p-1}$ to represent the intervals $\{[d_i, d_{i+1})\}_{i=1}^{p-1}$.
The corresponding adjacency matrix $S$ generates a subshift of finite type,  which we denote it by $\Sigma$, i.e. 
$$ \Sigma=\{({i_n})\in \{1,2,\cdots, p-1\}^{\mathbb{N}}:S_{{i_n}, {i_{n+1}}}=1\}.$$
Let  $$P=\{(i_n)\in \Sigma: \exists\, k\geq 0\, \mbox{with}\,\sigma^{k}(i_n)=(p-1)^{\infty}\}.$$
In this section, we always assume that the binary expansion of $x\in[0,1)$ is generated by the doubling map, i.e. any binary expansion $(a_n)\in\{0,1\}^{\mathbb{N}}$ cannot end with $1^{\infty}$.  Hence we should remove $P$ from $\Sigma$.

Note that for each $1\leq i\leq p-1$, it has an associated $B_{i}$ which is one of the sets of the Markov partition. 
For any $x\in [0,1)$, we can find its associated coding in $\Sigma$ using the orbit of $x$, i.e. we can find some sequence $(j_k)\in \Sigma\setminus P$ such that $x\in B_{j_1}, T(x)\in B_{j_2}, \cdots, T^k(x)\in B_{j_{k+1}},\cdots$. Conversely, for any $(j_k)\in \Sigma \setminus P$, we may also find some point $x\in [0,1)$ such that $T^k(x)\in B_{j_{k+1}}, k\geq 0$. 

The following lemma is clear. 
\setcounter{Lemma}{5}
\begin{Lemma}\label{[a,b)}
$\dim_{H}(J(a,b))=\dim_{H}(J[a,b)),$ where $$J[a,b)=\{x\in [0,1):T^{n}(x)\notin [a,b), \forall\, n\geq 0\},$$ moreover, $J(a,b)=J[a,b)$ except for a countable set. 
\end{Lemma}
\begin{proof}
The lemma follows the following simple inclusions.
$$J[a,b)\subset J(a,b)\subset  \cup_{n=1}^{\infty}\{x\in[0,1):T^n(x)=a\} \cup J[a,b).$$ 
\end{proof}
Since the Hausdorff dimension of a countable set is zero, in the remaining of this section, we only consider the set $J[a,b)$ instead of $J(a,b)$.

By virtue of  the definition of  Markov partition, the hole $[a,b)$ is the union of some  sets of this partition.  We denote its associated blocks by $B_{i_1}, B_{i_2},\cdots,B_{i_s}$.  For simplicity we denote these blocks by $\widehat{B_1}=B_{i_1}, \cdots, \widehat{B_s}=B_{i_s}$.

Let $S$ be the  adjacency matrix, and the associated blocks of the hole  are $B_{i_1}, B_{i_2},\cdots, B_{i_s}$. We remove the $i_j$-th row and $i_j$-th column of $S$, $1\leq j\leq s$, and denote this new matrix  by $S^{'}.$ Similarly, the subshift of finite type generated by $S^{'}$ is denoted by $\Sigma^{'}$. 
For  Example \ref{example} we mentioned above,
$$
 S^{'}=\begin{pmatrix}
  1 & 1 & 1  & 0& 0\\
  0 & 0 & 0  & 0& 0\\  
0 & 0 & 0  & 0& 0\\  
  0 & 1 & 0  & 0& 0\\
  0& 0 & 1 & 1& 1\\
 \end{pmatrix}.
$$
\setcounter{Lemma}{6}
\begin{Lemma}\label{bijection}
There is a bijection between  the set of all binary expansions of points in $[0,1)$ and $\Sigma\setminus{P}$. 
\end{Lemma}
\setcounter{Remark}{10}
\begin{proof}
Given $({i_n})\in \Sigma\setminus{P}$, for $i_1$ we can find a corresponding interval $B_{i_1}$. $B_{i_1}$ is in the domain of $T_0=2x$ or $T_1=2x-1$. If $B_{i_1}$ is in the domain of $T_0$, then we let $a_1=0$, otherwise, we set $a_1=1$. 
By the definition of Markov partition, $ B_{i_{n+1}}\subset T( B_{i_n})$ for any $n\geq 1$. Then similarly 
we  identify $({i_n})$ with  a sequence $(a_n)\in\{0,1\}^{\mathbb{N}}$ in the following way
\[a_n= \begin{cases} 
       0 &\mbox{if } B_{i_n} \mbox{ is in the domain of }T_0=2x\\
      1 & \mbox{if } B_{i_n} \mbox{ is in the domain of }T_1=2x-1
   \end{cases}
\]
Define $x=\sum_{n=1}^{\infty}a_n2^{-n}$, we know that $(a_n)$ is a binary expansion of $x$, which cannot end with $1^{\infty}$ as  $({i_n})\notin P$.
Conversely,  given any binary coding $(a_n)$ which does not end with $1^{\infty}$, we can define a point $x=\sum_{n=1}^{\infty}a_n2^{-n}$. For any $k$, $T^k(x)$ falls into some $B_j, 1\leq j\leq p-1$. Therefore we may find a unique
 associated sequence $({i_n})\in\Sigma\setminus{P}$.
 Now we can define a bijection $$\phi: \{\mbox{all the binary expansions under the doubling map}\}\to \Sigma\setminus{P}$$by 
 $$\phi(a_n)=({i_n}),$$
 if $T^{n}(x)\in B_{i_{n+1}}$, $n\geq 0$, where $x=\sum_{n=1}^{\infty}a_n2^{-n}$. 
\end{proof}

 In order to construct an isomorphism between $J[a,b)$ and $\Sigma^{'},$ we have to remove some points from these two sets respectively. 
Let $E=\{d_1=0,d_2,\cdots,d_{p-1}\}$,  define $$J[a,b)\setminus(\cup_{n=0}^{\infty} T^{-n}(E)).$$
By Lemma \ref{bijection}, every point in $x\in [0,1)$ has a unique coding in $\Sigma\setminus P$. Since $\cup_{n=0}^{\infty} T^{-n}(E)$ is a countable set, the corresponding codings of this set in $\Sigma\setminus P$ is also a countable set. Denote this set by $Q$. 

For any $x\in J[a,b)\setminus(\cup_{n=0}^{\infty} T^{-n}(E))$, we consider the orbit of $x$, i.e. we can find an infinite sequence $\{B_{j_k}\}^{\infty}_{k=1}$ such that $x\in B_{j_1}, T(x)\in B_{j_2}, \cdots, T^k(x)\in B_{j_{k+1}},\cdots$. Hence,  we obtain a unique $(j_k)\in \Sigma^{'}\setminus(P\cup Q)$. Subsequently we can define a map $$\phi: J[a,b)\setminus(\cup_{n=0}^{\infty} T^{-n}(E))\to \Sigma^{'}\setminus(P\cup Q)$$ by 
$$\phi(x)=(j_k).$$
Since $\Sigma^{'}$ is a subshift of finite type,  and the matrix $S^{'}$ is irreducible, 
we can define a Parry measure $\mu$, which is the unique measure of maximal entropy,  on $\Sigma^{'}$, see \cite{Par64}.   Now we can prove following result. 
\setcounter{Theorem}{7}
\begin{Theorem}\label{isomorphic}
Let $J[a,b)=\{x\in [0,1):T^{n}(x)\notin [a,b), \forall\, n\geq 0\},$  and $\mu$ is the Parry  measure of $\Sigma^{'}$. Then 
$(J[a,b), T, \mu\circ \phi)$ is measure theoretically isomorphic to $(\Sigma^{'}, \sigma, \mu)$.
\end{Theorem}
\begin{proof}
It is easy to check that $\sigma\circ \phi=\phi\circ T$. Hence, it remains to prove that $\phi$ is a bijection between $J[a,b)\setminus(\cup_{n=0}^{\infty} T^{-n}(E))$ and $ \Sigma^{'}\setminus(P\cup Q)$. Firstly $\phi$ is one-to-one. For any $x,y\in J[a,b)\setminus(\cup_{n=0}^{\infty} T^{-n}(E))$, 
if $\phi(x)=\phi(y)$, then we can  implement the idea, which is used in the proof of Lemma \ref{bijection}, to find the binary codings of $x$ and $y$. Since $\phi(x)=\phi(y)$, it follows that their associated binary codings also coincide. Hence, $x=y$.  On the other hand,   for any $(j_k)\in  \Sigma^{'}\setminus(P\cup Q)$, we can find a point $x$ such that $T^k(x)\in B_{j_{k+1}}, k\geq 0$ (also see the proof of Lemma \ref{bijection}). By the definition of  $ \Sigma^{'}\setminus(P\cup Q)$, $T^k(x)$ cannot hit  $\cup_{n=0}^{\infty} T^{-n}(E)$, and $\{B_{j_{k+1}},k\geq0\}$ does not consist of any $\widehat{B_1}, \cdots, \widehat{B_s}$. Therefore, $x\in J[a,b)\setminus(\cup_{n=0}^{\infty} T^{-n}(E))$. 
\end{proof}
Now Theorem \ref{subshift} follows from Theorem \ref{isomorphic} and Lemma \ref{[a,b)}.

We end with an example.
\setcounter{Example}{8}
\begin{Example}
Let $a=\dfrac{1}{31}$ and $b=\dfrac{2}{31}$. Then $\dim_{H}(J(a,b))=\dfrac{\log \alpha}{\log 2}$, where 
$\alpha$ is the largest positive root  of $x^4=x^3+x^2+x+1$.  
Note that the orbits of $a$ and $b$ are eventually periodic under the doubling map. 
Let $A=[0, 1/31),B=[1/31, 2/31),C=[2/31, 4/31),D=[4/31, 8/31),E=[8/31, 1/2),$
$F=[1/2,16/31), G=[16/31, 1)$. Then we can define an adjacency matrix
$$
 S^{'}=\begin{pmatrix}
  1 & 0& 0 & 0 & 0&0 \\
 0 & 0 & 1 & 0 & 0&0 \\ 
  0 & 0 & 0 & 1 & 1&0 \\
0& 0 & 0 & 0 & 0&1 \\
1& 0 & 0 & 0 & 0&0 \\
 0 & 1 & 1 & 1 & 1&1\\ 
 \end{pmatrix}
$$
with respect to $J(a,b)$. The spectral radius of this matrix is $\alpha.$
\end{Example} 
\section*{Acknowledgment}
The authors would like to thank the anonymous referees for many useful suggestions and remarks. 
The second author was supported by a grant from the China Scholarship Council grant no.201206140003, and by a grant from the
National Natural Science Foundation of China grant no.11271137, no.11671147.

\end{document}